\documentclass[12pt]{amsart}
\usepackage{amsmath,amsthm,amsfonts,amssymb,mathrsfs}
\date{\today}

\usepackage{cite}
\usepackage{color}
\usepackage{hyperref}

\newtheorem{theorem}{Theorem}[section]

\newtheorem{proposition}[theorem]{Proposition}
\newtheorem{corollary}[theorem]{Corollary}
\newtheorem{lemma}[theorem]{Lemma}
\theoremstyle{definition}
\newtheorem{remark}[theorem]{Remark}

\begin{document}

\title[THE MONOID OF ORDER ISOMORPHISMS BETWEEN PRINCIPAL FILTERS OF $\sigma{\mathbb{N}^\kappa}$...]{THE MONOID OF ORDER ISOMORPHISMS BETWEEN PRINCIPAL FILTERS OF $\sigma{\mathbb{N}^\kappa}$}

\author[Taras~Mokrytskyi]{Taras~Mokrytskyi}
\address{Faculty of Mathematics, National University of Lviv,
Universytetska 1, Lviv, 79000, Ukraine}
\email{tmokrytskyi@gmail.com}

\keywords{Semigroup, inverse semigroup, partial map, permutation group, least group congruence, bicyclic monoid, semidirect product}

\subjclass[2020]{20M18, 20M20, 20M30.}

\begin{abstract}
  Consider the following generalization of the bicyclic monoid. Let $\kappa$ be any infinite cardinal and let $\mathcal{IP\!F}\left(\sigma{\mathbb{N}^\kappa}\right)$ be the semigroup of all order isomorphisms between principal filters of the set $\sigma{\mathbb{N}^\kappa}$ with the product order. We shall study algebraic properties of the semigroup $\mathcal{IP\!F}\left(\sigma{\mathbb{N}^\kappa}\right)$, show that it is bisimple, $E$-unitary, $F$-inverse semigroup, describe Green's relations on $\mathcal{IP\!F}\left(\sigma{\mathbb{N}^\kappa}\right)$, describe the group of units $H\left(\mathbb{I}\right)$ of the semigroup $\mathcal{IP\!F}\left(\sigma{\mathbb{N}^\kappa}\right)$ and describe its maximal subgroups. We prove that
  the semigroup $\mathcal{IP\!F}\left(\sigma{\mathbb{N}^\kappa}\right)$ is isomorphic to the semidirect product $\mathcal{S}_\kappa\ltimes\sigma{\mathbb{B}^\kappa}$ of the semigroup $\sigma{\mathbb{B}^\kappa}$ by the group $\mathcal{S}_\kappa$, show that
  every non-identity congruence $\mathfrak{C}$ on the semigroup $\mathcal{IP\!F}\left(\sigma{\mathbb{N}^\kappa}\right)$ is a group congruence and describe the least group congruence on $\mathcal{IP\!F}\left(\sigma{\mathbb{N}^\kappa}\right)$.
\end{abstract}

\maketitle

\section{\bf Introduction and preliminaries}\label{S:intro}

In this paper, we shall denote
the set of integers by $\mathbb{Z}$,
the set of positive integers by $\mathbb{N}$,
the set of all maps from cardinal $\kappa$ to the set $X$ by $X^\kappa$
and the symmetric group of degree $\kappa$ by $\mathcal{S}_\kappa$, i.e., $\mathcal{S}_\kappa$ is the group of all bijections of the set $\kappa$.
For set $X$, by $id_X$ we denote the identity map $id_X\colon X \to X$, $id_X\colon x\mapsto x$ for any $x\in X$.
For map $f\colon X\to Y$ and for subset $A\subset X$ we denote $\left(A\right)f = \{\left(x\right)f \mid x\in X\}$.

Let $\left(X,\leqslant\right)$ be a partially ordered set (a poset). For an arbitrary $x\in X$ we denote
 \begin{equation*}
{\uparrow}x=\left\{y\in X\colon x\leqslant y\right\} \qquad \hbox{and} \qquad {\downarrow}x=\left\{y\in X\colon y\leqslant x\right\}.
\end{equation*}
The sets ${\uparrow}x$ and ${\downarrow}x$ are called the \emph{principal filter} and the \emph{principal ideal}, respectively, generated by the element $x\in X$. A map $\alpha\colon \left(X,\leqslant\right)\to\left(Y,\eqslantless\right)$ from poset $\left(X,\leqslant\right)$ into a poset $\left(Y,\eqslantless\right)$ is called \emph{monotone} or \emph{order preserving} if $x\leqslant y$ in $\left(X,\leqslant\right)$ implies that $x\alpha\eqslantless y\alpha$ in $\left(Y,\eqslantless\right)$. A monotone map $\alpha\colon \left(X,\leqslant\right)\to\left(Y,\eqslantless\right)$ is said to be \emph{order isomorphism} if it is bijective and its converse $\alpha^{-1}\colon\left(Y,\eqslantless\right)\to\left(X,\leqslant\right)$ is monotone.

An semigroup $S$ is called {\it inverse} if for any element $x\in S$ there exists a unique $x^{-1}\in S$ such that $xx^{-1}x=x$ and $x^{-1}xx^{-1}=x^{-1}$. The element $x^{-1}$ is called the {\it inverse of} $x\in S$. If $S$ is an inverse semigroup, then the function $\operatorname{inv}\colon S\to S$ which assigns to every element $x$ of $S$ its inverse element $x^{-1}$ is called the {\it inversion}.

If $S$ is a semigroup, then we shall denote the subset of all idempotents in $S$ by $E\left(S\right)$. If $S$ is an inverse semigroup, then $E\left(S\right)$ is closed under multiplication. The semigroup operation on $S$ determines the following partial order $\preccurlyeq$
on $E\left(S\right)$: $e\preccurlyeq f$ if and only if $ef=fe=e$. This order is called the {\em natural partial order} on $E\left(S\right)$. A \emph{semilattice} is a commutative semigroup of idempotents.

If $S$ is a semigroup, then we shall denote the Green relations on $S$ by $\mathscr{R}$, $\mathscr{L}$, $\mathscr{J}$, $\mathscr{D}$ and $\mathscr{H}$ (see \cite{Clifford-Preston-1961-1967}). A semigroup $S$ is called \emph{simple} if $S$ does not contain proper two-sided ideals and \emph{bisimple} if $S$ has only one $\mathscr{D}$-class.

Hereafter we shall assume that $\lambda$ is an infinite cardinal.
If $\alpha\colon \lambda\rightharpoonup \lambda$ is a partial map, then we shall denote
the domain and the range of $\alpha$ by $\operatorname{dom}\alpha$ and $\operatorname{ran}\alpha$, respectively.

Let $\mathscr{I}_\lambda$ be the set of all partial one-to-one
transformations of a cardinal $\lambda$
together with the following semigroup operation:
\begin{equation*}
x\left(\alpha\beta\right)=\left(x\alpha\right)\beta \quad \hbox{if} \quad
x\in\operatorname{dom}\left(\alpha\beta\right)=\left\{
y\in\operatorname{dom}\alpha\mid
y\alpha\in\operatorname{dom}\beta\right\},\quad \hbox{for }
\alpha,\beta\in\mathscr{I}_\lambda.
\end{equation*}
The semigroup $\mathscr{I}_\lambda$ is called the \emph{symmetric
inverse semigroup} over the cardinal $\lambda$~(see \cite[Section~1.9]{Clifford-Preston-1961-1967}).
The symmetric inverse semigroup was introduced by
Wagner~\cite{Vagner-1952} and it plays a major role in the theory of
semigroups.

The \emph{bicyclic semigroup} (or the \emph{bicyclic monoid}) ${\mathscr{C}}\left(p,q\right)$ is the semigroup with the identity $1$ generated by elements $p$ and $q$ subject only to the condition $pq=1$.

The bicyclic semigroup plays an important role in the algebraic theory of semigroups and the theory of topological semigroups. For instance, a well-known Andersen's result~\cite{Andersen-1952} states that a ($0$-)simple semigroup with an idempotent is completely ($0$-)simple if and only if it does not contain an isomorphic copy of the bicyclic semigroup.

The bicyclic monoid admits only the discrete semigroup topology. Bertman and  West in \cite{Bertman-West-1976} extended this result for the case of semitopological semigroups. Stable and $\Gamma$-compact topological semigroups do not contain the bicyclic monoid~\cite{Anderson-Hunter-Koch-1965, Hildebrant-Koch-1986}. The problem of an embedding of the bicyclic monoid into compact-like topological semigroups was studied in \cite{Banakh-Dimitrova-Gutik-2009, Banakh-Dimitrova-Gutik-2010, Gutik-Repovs-2007}. The study of various generalizations of the bicyclic monoid, their algebraic and topological properties, like topologizations, shift-continuous topologizations and embedding into compact-like topological semigroups was conducted in several publications, including \cite{Bardyla-Gutik-2016, Bardyla-Gutik-2020, Chuchman-Gutik-2010, Gutik-Lysetska-2023, Gutik-Khylynskyi-2022, Gutik-Krokhmalna-2019, Gutik-Maksymyk-2016, Gutik-Maksymyk-2016-2, Gutik-Mokrytskyi-2020, Gutik-Mykhalenych-2020, Gutik-Popadiuk-2022, Gutik-Savchuk-2018, Gutik-Savchuk-2019, Mokrytskyi-2019, Gutik-Pozdniakova-2022}.

\begin{remark}\label{remark-01.1}
We observe that the bicyclic semigroup is isomorphic to the
semigroup $\mathscr{C}_{\mathbb{N}}\left(\alpha,\beta\right)$ which is
generated by partial transformations $\alpha$ and $\beta$ of the set
of positive integers $\mathbb{N}$, defined as follows:
$\left(n\right)\alpha=n+1$ if $n\geq 1$ and $\left(n\right)\beta=n-1$ if $n> 1$ (see Exercise~IV.1.11$\left(ii\right)$ in \cite{Petrich-1984}).
\end{remark}

Taking into account this remark, we shall consider the following generalization of the bicyclic semigroup. For an arbitrary positive integer $n\geq 2$ by $\left(\mathbb{N}^n,\leqslant\right)$ we denote the $n$-th power of the set of positive integers $\mathbb{N}$ with the product order:
\begin{equation*}
  \left(x_1,\ldots,x_n\right)\leqslant\left(y_1,\ldots,y_n\right) \qquad \hbox{if and only if} \qquad x_i\leq y_i \quad \hbox{for all} \quad i=1,\ldots,n.
\end{equation*}
It is obvious that the set of all order isomorphisms between principal filters of the poset $\left(\mathbb{N}^n,\leqslant\right)$ with the operation of the composition of partial maps forms a semigroup. Denote this semigroup by $\mathcal{IP\!F}(\mathbb{N}^n)$. The structure of the semigroup $\mathcal{IP\!F}(\mathbb{N}^n)$ was introduced and studied in \cite{Gutik-Mokrytskyi-2020}. There was shown that $\mathcal{IP\!F}(\mathbb{N}^n)$ is a bisimple, $E$-unitary, $F$-inverse monoid, described Green's relations on $\mathcal{IP\!F}(\mathbb{N}^n)$ and its maximal subgroups. It was proved that $\mathcal{IP\!F}(\mathbb{N}^n)$ is isomorphic to the semidirect product of the direct $n$-th power of the bicyclic monoid ${\mathscr{C}}^n(p,q)$ by the group of permutation $\mathcal{S}_n$, every non-identity congruence on  $\mathcal{IP\!F}(\mathbb{N}^n)$ is group and was described the least group congruence on $\mathcal{IP\!F}(\mathbb{N}^n)$. It was shown that every shift-continuous topology on $\mathcal{IP\!F}(\mathbb{N}^n)$ is discrete and discussed embedding of the semigroup $\mathcal{IP\!F}(\mathbb{N}^n)$ into compact-like topological semigroups. In \cite{Mokrytskyi-2019} it was proved that a Hausdorff locally compact semitopological semigroup $\mathcal{IP\!F}(\mathbb{N}^n)$ with an adjoined zero is either compact or discrete. In this paper we shall extend this generalization from $\mathbb{N}^n$ to $\sigma{\mathbb{N}^\kappa}$ for any infinite cardinal $\kappa$.

For any infinite cardinal $\kappa$ consider the subset $\sigma{\mathbb{N}^\kappa}$ of $\mathbb{N}^\kappa$ which contains all maps $a$ such that the set $\{ x\in \kappa \mid \left(x\right)a \neq 1 \}$ is finite, i.e.,
\begin{equation*}
  \sigma{\mathbb{N}^\kappa} = \{ a \in \mathbb{N}^\kappa \mid  \{ x\in \kappa \mid \left(x\right)a \neq 1 \} \hbox{ is finite } \}.
\end{equation*}

Similarly define $\sigma{\mathbb{Z}^\kappa}$ as the subset of $\mathbb{Z}^\kappa$ which contains all maps $a$ such that the set $\{ x\in \kappa \mid \left(x\right)a \neq 0 \}$ is finite.

By $\mathbf{1}$ we shall denote the element of the $\mathbb{N}^\kappa$ such that $\left(x\right)\mathbf{1}=1$ for any $x\in\kappa$.

On the set $\mathbb{Z}^\kappa$ consider the product order $\leqslant$:
\begin{equation*}
  a\leqslant b \qquad \hbox{if and only if} \qquad \left(x\right)a\leq \left(x\right)b \quad \hbox{for all} \quad x\in\kappa.
\end{equation*}

Also, consider the pointwise operations $+$, $-$, $\operatorname{max}$ and $\operatorname{min}$ on the set $\mathbb{Z}^\kappa$. For any $a,b \in \mathbb{Z}^\kappa$ define
\begin{equation*}
  \begin{split}
  &\left(x\right)\left(a+b\right) = \left(x\right)a+\left(x\right)b,\\
  &\left(x\right)\left(a-b\right) = \left(x\right)a-\left(x\right)b, \\
  & \left(x\right)\left(\operatorname{max}\{a,b\}\right)=\operatorname{max}\{\left(x\right)a,\left(x\right)b\}, \\
  & \left(x\right)\left(\operatorname{min}\{a,b\}\right)=\operatorname{min}\{\left(x\right)a,\left(x\right)b\}
  \end{split}
\end{equation*}
for any $x\in \kappa$.
 It is obvious that the set $\sigma{\mathbb{Z}^\kappa}$ is closed under these operations. The set $\sigma{\mathbb{N}^\kappa}$ is also closed under the operation $\operatorname{max}$ and $\operatorname{min}$ but not for $+$ and $-$. Moreover
\begin{equation*}
  a+b, a-b \notin \sigma{\mathbb{N}^\kappa} \qquad \hbox{for any} \qquad a, b \in \sigma{\mathbb{N}^\kappa}.
\end{equation*}
But
\begin{equation*}
  a + b - \mathbf{1} \in \sigma{\mathbb{N}^\kappa} \qquad \hbox{for any} \qquad a, b \in \sigma{\mathbb{N}^\kappa},
\end{equation*}
and
\begin{equation*}
  a-b+\mathbf{1}\in \sigma{\mathbb{N}^\kappa} \qquad \hbox{for any} \qquad a\in \sigma{\mathbb{N}^\kappa} \qquad \hbox{and} \qquad b\in {\downarrow}a.
\end{equation*}

Let $\kappa$ by any infinite cardinal. Define the semigroup $\mathcal{IP\!F}\left(\sigma{\mathbb{N}^\kappa}\right)$ as the set of all order isomorphisms between principal filters of the poset $\left(\sigma{\mathbb{N}^\kappa}, \leqslant\right)$ with the operation of the composition of partial maps, i.e.,
\begin{equation*}
\mathcal{IP\!F}\left(\sigma{\mathbb{N}^\kappa}\right)=\left(\{\alpha\colon {\uparrow}a\to {\uparrow}b \mid a,b\in \sigma{\mathbb{N}^\kappa} \hbox{ and } \alpha \text{ is an order isomorphism}\}, \circ\right).
\end{equation*}

Consider the following notation.
For any $\alpha \in \mathcal{IP\!F}\left(\sigma{\mathbb{N}^\kappa}\right)$ by $d_\alpha$ and $r_\alpha$ we denote the elements of $\sigma{\mathbb{N}^\kappa}$ such that $\operatorname{dom}{\alpha}={\uparrow}d_\alpha$ and $\operatorname{ran}{\alpha}={\uparrow}r_\alpha$

Also we define the maps $\lambda_\alpha, \rho_\alpha \in \mathcal{IP\!F}\left(\sigma{\mathbb{N}^\kappa}\right)$ in the following way:
\begin{equation*}
\begin{split}
\operatorname{dom}\rho_\alpha=\operatorname{dom}\alpha={\uparrow}d_\alpha, \quad \operatorname{ran}\rho_\alpha=\sigma{\mathbb{N}^\kappa}, \quad \left(a\right)\rho_\alpha=a-d_\alpha+\mathbf{1}  \quad \hbox{ for } a\in \operatorname{dom}\rho_\alpha; \\
\operatorname{ran}\lambda_\alpha=\operatorname{ran}\alpha={\uparrow}r_\alpha, \quad \operatorname{dom}\lambda_\alpha=\sigma{\mathbb{N}^\kappa}, \quad \left(a\right)\lambda_\alpha=a+r_\alpha-\mathbf{1}  \quad \hbox{ for } a\in \operatorname{dom}\lambda_\alpha.
\end{split}
\end{equation*}
Since $a+r_\alpha-\mathbf{1} \in \sigma{\mathbb{N}^\kappa}$ for any $a\in \operatorname{dom}\lambda_\alpha$ we have that $\lambda_\alpha$ is well-defined.
Similarly, $a-d_\alpha+\mathbf{1} \in \sigma{\mathbb{N}^\kappa}$ for any $a\in \operatorname{dom}\rho_\alpha$, so $\rho_\alpha$ is well-defined too.
We note that the definition of $\lambda_\alpha, \rho_\alpha$ implies that $\lambda_{\lambda_\alpha}=\lambda_\alpha$ and $\rho_{\rho_\alpha} = \rho_\alpha$.

For any infinite cardinal $\kappa$ and for any bijection $g\in \mathcal{S}_\kappa$ define the selfmap $\mathcal{F}_g\colon\mathbb{Z}^\kappa\to\mathbb{Z}^\kappa$ by formula:
  \begin{equation*}
  \left(x\right)\left(a\right)\mathcal{F}_g=\left(\left(x\right)g^{-1}\right)a, \text{ } a\in\mathbb{Z}^\kappa, \text{ } x\in \kappa.
  \end{equation*}

  \section{\bf Algebraic properties of the semigroup $\mathcal{IP\!F}\left(\sigma{\mathbb{N}^\kappa}\right)$}

  \begin{proposition}\label{proposition-2.1}
    For any infinite cardinal $\kappa$ the following statements hold:
    \begin{itemize}
      \item[$\left(i\right)$] $\mathcal{IP\!F}\left(\sigma{\mathbb{N}^\kappa}\right)$ is an inverse semigroup;
      \item[$\left(ii\right)$] the semilattice $E\left(\mathcal{IP\!F}\left(\sigma{\mathbb{N}^\kappa}\right)\right)$ is isomorphic to the semilattice $\left(\sigma{\mathbb{N}^\kappa},\operatorname{max}\right)$ by the mapping $\varepsilon \mapsto d_\varepsilon$;
      \item[$\left(iii\right)$] $\alpha\mathscr{L}\beta$ in $\mathcal{IP\!F}\left(\sigma{\mathbb{N}^\kappa}\right)$ if and only if $\operatorname{dom}\alpha=\operatorname{dom}\beta$;
      \item[$\left(iv\right)$] $\alpha\mathscr{R}\beta$ in $\mathcal{IP\!F}\left(\sigma{\mathbb{N}^\kappa}\right)$ if and only if $\operatorname{ran}\alpha=\operatorname{ran}\beta$;
      \item[$\left(v\right)$] $\alpha\mathscr{H}\beta$ in $\mathcal{IP\!F}\left(\sigma{\mathbb{N}^\kappa}\right)$ if and only if $\operatorname{dom}\alpha=\operatorname{dom}\beta$ and $\operatorname{ran}\alpha=\operatorname{ran}\beta$;
      \item[$\left(vi\right)$] for any idempotents $\varepsilon,\iota\in \mathcal{IP\!F}\left(\sigma{\mathbb{N}^\kappa}\right)$ there exist elements $\alpha,\beta\in
        \mathcal{IP\!F}\left(\sigma{\mathbb{N}^\kappa}\right)$ such that $\alpha\beta=\varepsilon$ and $\beta\alpha=\iota$, hence $\mathcal{IP\!F}\left(\sigma{\mathbb{N}^\kappa}\right)$ is bisimple which implies that it is simple.
    \end{itemize}
    \end{proposition}

    \begin{proof}[\textsl{Proof}]
      $\left(i\right)$ The definition of the semigroup $\mathcal{IP\!F}\left(\sigma{\mathbb{N}^\kappa}\right)$ implies that $\mathcal{IP\!F}\left(\sigma{\mathbb{N}^\kappa}\right)$ is an inverse subsemigroup of the symmetric inverse monoid $\mathcal{I}_{\sigma{\mathbb{N}^\kappa}}$ over the set $\sigma{\mathbb{N}^\kappa}$.

      $\left(ii\right)$ implies from statement $\left(i\right)$.

      $\left(iii\right)$--$\left(v\right)$ follow from statement $\left(i\right)$ and Proposition~3.2.11$\left(1\right)$--$\left(3\right)$ of \cite{Lawson-1998}.

      $\left(vi\right)$ Fix arbitrary idempotents $\varepsilon,\iota\in\mathcal{IP\!F}\left(\sigma{\mathbb{N}^\kappa}\right)$. Define a partial map $\alpha\colon \sigma{\mathbb{N}^\kappa}\rightharpoonup \sigma{\mathbb{N}^\kappa}$ in the following way:
    \begin{equation*}
      \operatorname{dom}\alpha=\operatorname{dom}\varepsilon, \qquad \operatorname{ran}\alpha=\operatorname{dom}\iota \qquad \hbox{and}\qquad \left(z\right)\alpha=z-d_\varepsilon+d_\iota, \quad \hbox{for any} \quad z\in\operatorname{dom}\alpha.
    \end{equation*}
    Since $\varepsilon,\iota\in\mathcal{IP\!F}\left(\sigma{\mathbb{N}^\kappa}\right)$, the partial map $\alpha$ is well-defined and $\alpha\in \mathcal{IP\!F}\left(\sigma{\mathbb{N}^\kappa}\right)$. Then $\alpha\alpha^{-1}=\varepsilon$ and $\alpha^{-1}\alpha=\iota$ and hence we put $\beta=\alpha^{-1}$. Lemma 1.1 from \cite{Munn-1966} implies that $\mathcal{IP\!F}\left(\sigma{\mathbb{N}^\kappa}\right)$ is bisimple and hence simple.
    \end{proof}

  For any positive integer $k\geq 2$ and for any $x \in \kappa$, consider the map $k_x\colon \kappa\to \mathbb{N}$ defined by
  \begin{equation*}
  \left(t\right)k_x =
   \begin{cases}
     k, &\text{if~~$t=x$,}\\
     1, &\text{otherwise.}
   \end{cases}
  \end{equation*}

  \begin{lemma}\label{lemma-1.2-3} For any infinite cardinal $\kappa$ and for any bijection $g\in \mathcal{S}_\kappa$, the following statements hold:
    \begin{itemize}
      \item[$\left(i\right)$] The selfmap $\mathcal{F}_g$ is an order automorphism of the poset $\left(\mathbb{Z}^\kappa, \leqslant\right)$, and $\left(\mathcal{F}_g\right)^{-1}=\mathcal{F}_{g^{-1}}$.
      \item[$\left(ii\right)$] $\left(\sigma{\mathbb{N}^\kappa}\right)\mathcal{F}_g$ = $\sigma{\mathbb{N}^\kappa}$.
      \item[$\left(iii\right)$] $\left(\sigma{\mathbb{Z}^\kappa}\right)\mathcal{F}_g$ = $\sigma{\mathbb{Z}^\kappa}$.
      \item[$\left(iv\right)$] $\mathcal{F}_{g h}=\mathcal{F}_g\mathcal{F}_h$ for any $h \in \mathcal{S}_\kappa$.
      \item[$\left(v\right)$] For any $k\in\mathbb{N}$ and for any $x \in \kappa \colon{} \left(k_x\right)\mathcal{F}_g=k_{\left(x\right)g}$.
      \item[$\left(vi\right)$] $\left(\mathbf{1}\right)\mathcal{F}_g=\mathbf{1}$.
      \item[$\left(vii\right)$] For any $h \in \mathcal{S}_\kappa \colon{} g\neq h \implies \mathcal{F}_g\neq \mathcal{F}_h$.

      \item[$\left(viii\right)$] For any $a, b\in \mathbb{Z}^\kappa \colon \left(a+b\right)\mathcal{F}_g=\left(a\right)\mathcal{F}_g + \left(b\right)\mathcal{F}_g$.
      \item[$\left(ix\right)$] For any $a, b\in \mathbb{Z}^\kappa \colon \left(a-b\right)\mathcal{F}_g=\left(a\right)\mathcal{F}_g - \left(b\right)\mathcal{F}_g$.
      \item[$\left(x\right)$] For any $a, b\in \mathbb{Z}^\kappa \colon \left(\operatorname{max}\{a,b\}\right)\mathcal{F}_g = \operatorname{max}\{\left(a\right)\mathcal{F}_g,\left(b\right)\mathcal{F}_g\}$.
      \item[$\left(xi\right)$] For any $a, b\in \mathbb{Z}^\kappa \colon \left(\operatorname{min}\{a,b\}\right)\mathcal{F}_g = \operatorname{min}\{\left(a\right)\mathcal{F}_g,\left(b\right)\mathcal{F}_g\}$.
    \end{itemize}
  \end{lemma}

  \begin{proof}[\textsl{Proof}]
    $\left(i\right)$
    Show that $\mathcal{F}_g$ is an order isomorphism.
    Fix distinct $a, b\in\mathbb{Z}^\kappa$. Then there exists $x\in \kappa$ such that $\left(x\right)a \neq \left(x\right)b$. For $y=\left(x\right)g,$ we have that $x=\left(y\right)g^{-1}$, then $\left(\left(y\right)g^{-1}\right)a \neq \left(\left(y\right)g^{-1}\right)b$ implies that $\left(a\right)\mathcal{F}_g\neq\left(b\right)\mathcal{F}_g,$ so $\mathcal{F}_g$ is injective.

    For any $a\in\mathbb{Z}^\kappa$, consider the map $b\colon \left(x\right)b=\left(\left(x\right)g\right)a$ for any $x\in \kappa$, then
    \begin{equation*}
    \left(x\right)\left(b\right)\mathcal{F}_g=\left(\left(x\right)g^{-1}\right)b=\left(\left(\left(x\right)g^{-1}\right)g\right)a=\left(x\right)a
    \end{equation*}
    for any $x\in \kappa$, so $\mathcal{F}_g$ is surjective and moreover its converse $\left(\mathcal{F}_g\right)^{-1}$ is equals to the $\mathcal{F}_{g^{-1}}$.

    Let $a,b\in\mathbb{Z}^\kappa$ and $a\leqslant b$. For any $x\in \kappa$ we have that $\left(\left(x\right)g^{-1}\right)a\leqslant \left(\left(x\right)g^{-1}\right)b$ which implies that $\left(x\right)\left(a\right)\mathcal{F}_g\leqslant\left(x\right)\left(b\right)\mathcal{F}_g$, i.e., $\left(a\right)\mathcal{F}_g\leqslant \left(b\right)\mathcal{F}_g,$ so $\mathcal{F}_g$ is monotone and such is $\mathcal{F}_g^{-1},$ therefore $\mathcal{F}_g$ is an order isomorphism.

    $\left(ii\right)$
  Fix an element $a \in \sigma{\mathbb{N}^\kappa}$. Since $\left(x\right)\left(a\right)\mathcal{F}_g = \left(\left(x\right)g^{-1}\right)a\in\mathbb{N}$ for any $x\in\kappa$ we have that $\left(a\right)\mathcal{F}_g\in \mathbb{N}^\kappa$. Consider the set $A=\{x\in \kappa \mid \left(x\right)a\neq 1\}$ and suppose that $\left(x\right)\left(a\right)\mathcal{F}_g\neq 1$ for some $x\in\kappa$, then $\left(\left(x\right)g^{-1}\right)a\neq 1$ and therefore $\left(x\right)g^{-1}\in A$, so $x\in \left(A\right)g$. Since the set $A$ is finite and $g$ is a bijection, we have that the set $\left(A\right)g$ is finite as well. So $\left(a\right)\mathcal{F}_g \in \sigma{\mathbb{N}^\kappa}$, therefore $\left(\sigma{\mathbb{N}^\kappa}\right)\mathcal{F}_g \subset \sigma{\mathbb{N}^\kappa}$.
  By proved above, we have that $\left(a\right)\mathcal{F}_{g^{-1}}\in \sigma{\mathbb{N}^\kappa}$, then $\left(\left(a\right)\mathcal{F}_{g^{-1}}\right)\mathcal{F}_g = a$ implies that $\sigma{\mathbb{N}^\kappa} \subset \left(\sigma{\mathbb{N}^\kappa}\right)\mathcal{F}_g$.

  $\left(iii\right)$
  The proof is similar to the proof of $\left(ii\right)$.

  $\left(iv\right)$
    For any $h\in \mathcal{S}_\kappa$, $a\in\mathbb{Z}^\kappa$ and $x\in \kappa$ we have that
    \begin{align*}
     \left(x\right)\left(a\right)\mathcal{F}_{g h}&= \left(\left(x\right)\left(gh\right)^{-1}\right)a= \\
       &= \left(\left(x\right)\left(h^{-1}g^{-1}\right)\right)a=\\
       &= \left(\left(\left(x\right)h^{-1}\right)g^{-1}\right)a=\\
       &= \left(\left(x\right)h^{-1}\right)\left(a\right)\mathcal{F}_{g}=\\
       &= \left(x\right)\left(\left(a\right)\mathcal{F}_g\right)\mathcal{F}_h =\\
       &= \left(x\right)\left(a\right)\left(\mathcal{F}_g\mathcal{F}_h\right).
    \end{align*}

  $\left(v\right)$
    Let $k\in \mathbb{N}$ and $x\in \kappa$. Then for any $t\in\kappa$ we have that
    \begin{align*}
      \left(t\right)\left(k_x\right)\mathcal{F}_g&= \left(\left(t\right)g^{-1}\right)k_x = \\
       &=\begin{cases}
         k, &\text{if $\left(t\right)g^{-1}=x$}\\
         1, &\text{otherwise}
       \end{cases}=\\
       &=\begin{cases}
        k, &\text{if $t=\left(x\right)g$}\\
        1, &\text{otherwise}
       \end{cases}=\\
       &=\left(t\right)k_{\left(x\right)g}.
    \end{align*}

  $\left(vi\right)$ For any $t \in \kappa $ we have that $\left(t\right)\left(\mathbf{1}\right)\mathcal{F}_g=\left(\left(t\right)g^{-1}\right)\mathbf{1}=1$.

  $\left(vii\right)$ Let $h\in\mathcal{S}_\kappa$ and $g\neq h$. Then there exists $x\in \kappa$ such that $\left(x\right)g^{-1}\neq \left(x\right)h^{-1}$. Consider the image of $2_{\left(x\right)g^{-1}}$ under the maps $\mathcal{F}_g$ and $\mathcal{F}_h$. Statement $\left(v\right)$ and the inequality $\left(x\right)g^{-1}\neq \left(x\right)h^{-1}$ imply that:
    \begin{equation*}
      \left(2_{\left(x\right)g^{-1}}\right)\mathcal{F}_g = 2_x \neq 2_{\left(\left(x\right)g^{-1}\right)h} = \left(2_{\left(x\right)g^{-1}}\right)\mathcal{F}_h.
    \end{equation*}

  $\left(viii\right)$ For any $a, b\in \mathbb{Z}^\kappa$ and for any $x\in \kappa$ we have that
  \begin{align*}
    \left(x\right)\left(a+b\right)\mathcal{F}_g& =\left(\left(x\right)g^{-1}\right)\left(a+b\right)= \\
     &= \left(\left(x\right)g^{-1}\right)a+\left(\left(x\right)g^{-1}\right)b= \\
     &= \left(x\right)\left(a\right)\mathcal{F}_g+\left(x\right)\left(b\right)\mathcal{F}_g.
  \end{align*}

    Proof of statements $\left(ix\right)$ and $\left(xi\right)$ are similar to the proof of $\left(viii\right)$.
  \end{proof}

  For any infinite cardinal $\kappa$ and for any bijection $g\in \mathcal{S}_\kappa$ define the map $\mathcal{F}^\circ_g\colon\sigma{\mathbb{N}^\kappa}\to\sigma{\mathbb{N}^\kappa}$ as the restriction of the map $\mathcal{F}_g$ to the set $\sigma{\mathbb{N}^\kappa}$. By statement $\left(ii\right)$ of Lemma~\ref{lemma-1.2-3}, the map $\mathcal{F}^\circ_g$ is well-defined and $\mathcal{F}^\circ_g$ is a bijection. This and statement $\left(i\right)$ of Lemma~\ref{lemma-1.2-3} imply that the map $\mathcal{F}^\circ_g$ is an order isomorphism of the poset $\left(\sigma{\mathbb{N}^\kappa}, \leqslant\right)$. Similarly, define the map $\mathcal{F}^\diamond_g\colon\sigma{\mathbb{Z}^\kappa}\to\sigma{\mathbb{Z}^\kappa}$ as the restriction of the map $\mathcal{F}_g$ to the set $\sigma{\mathbb{Z}^\kappa}$. And similarly, statement $\left(iii\right)$ of Lemma~\ref{lemma-1.2-3} implies that the map $\mathcal{F}^\diamond_g$ is well-defined and $\mathcal{F}^\diamond_g$ is a bijection. 
  
  The proof to the next lemma is similar to the proof of Lemma \ref{lemma-1.2-3}.

  \begin{lemma}\label{lemma-1.2-3-next} For any infinite cardinal $\kappa$ and for any bijection $g\in \mathcal{S}_\kappa$
  statements $\left(iv\right)-\left(xi\right)$ of Lemma \ref{lemma-1.2-3} also hold for $\mathcal{F}^\circ_g$ and $\mathcal{F}^\diamond_g$.
  \end{lemma}

We shall denote by $\mathbb{I}$  the identity map of $\sigma{\mathbb{N}^\kappa}$. It is obvious that $\mathbb{I}$ is the unit element of the semigroup $\mathcal{IP\!F}\left(\sigma{\mathbb{N}^\kappa}\right)$. Also by $H\left(\mathbb{I}\right)$ we shall denote the group of units of $\mathcal{IP\!F}\left(\sigma{\mathbb{N}^\kappa}\right)$. It is clear that $\alpha\in \mathcal{IP\!F}\left(\sigma{\mathbb{N}^\kappa}\right)$ is an element of $H\left(\mathbb{I}\right)$ if and only if it is an order isomorphism of the poset $\left(\sigma{\mathbb{N}^\kappa},\leqslant\right)$.

  \begin{lemma}\label{lemma-1.1}
    Let $\kappa$ be any infinite cardinal and $\alpha \in H\left(\mathbb{I}\right)$. Then $\left(\mathbf{1}\right)\alpha=\mathbf{1}$ and for any $x\in \kappa$ there exists $y\in \kappa$ such that $\left(k_x\right)\alpha = k_y$ for any positive integer $k \geq 2$.
  \end{lemma}

  \begin{proof}[\textsl{Proof}]
  Consider $\left(\mathbf{1}\right)\alpha$. Statement $\mathbf{1} \leqslant \left(\mathbf{1}\right)\alpha$ implies that $\left(\mathbf{1}\right)\alpha^{-1} \leqslant \left(\left(\mathbf{1}\right)\alpha\right)\alpha^{-1} = \mathbf{1}$, so $\left(\mathbf{1}\right)\alpha = \mathbf{1}$.

  Now, consider any $x\in \kappa$ and consider $\left(2_x\right)\alpha$. Since $\mathbf{1} = \left(\mathbf{1}\right)\alpha \neq \left(2_x\right)\alpha$, there exists $y\in \kappa$ such that $2_y \leqslant \left(2_x\right)\alpha,$ and the inequality $\left(2_y\right)\alpha^{-1} \leqslant 2_x$ implies that $\left(2_x\right)\alpha = 2_y.$

  Let $k\geq 2$ be a positive integer, suppose that for any positive integer $n\leq k$ the statement of the lemma holds.

  For any $x \in \kappa$ consider the image $\left(\left(k+1\right)_x\right)\alpha$. There exists $z\in \kappa$ such that $\left(k+1\right)_z \leqslant \left(\left(k+1\right)_x\right)\alpha$. Suppose the contrary that $\left(k+1\right)_z \nleqslant \left(\left(k+1\right)_x\right)\alpha$ for any $z\in \kappa$. Since
  \begin{equation*}
    \left(\left(k+1\right)_x\right)\alpha \notin\{ \mathbf{1}, 2_z, 3_z, \dots, k_z \mid z\in \kappa \},
  \end{equation*}
  there exist two distinct elements $z_1, z_2\in \kappa$ such that
  \begin{equation*}
  1 < \left(z_1\right)\left(\left(k+1\right)_x\right)\alpha < k+1 \quad \hbox{and} \quad 1 < \left(z_2\right)\left(\left(k+1\right)_x\right)\alpha < k+1.
  \end{equation*}
  Hence we have that
  \begin{equation*}
  2_{z_1}\leqslant\left(\left(k+1\right)_x\right)\alpha \quad \hbox{and} \quad 2_{z_2}\leqslant\left(\left(k+1\right)_x\right)\alpha,
  \end{equation*}
  and then
  \begin{equation*}
  \left(2_{z_1}\right)\alpha^{-1}\leqslant\left(k+1\right)_x \quad \hbox{and} \quad  \left(2_{z_2}\right)\alpha^{-1}\leqslant\left(k+1\right)_x.
  \end{equation*}
  Since $\left(2_{z_1}\right)\alpha^{-1} = 2_{z_1^\prime}$ and $\left(2_{z_2}\right)\alpha^{-1} = 2_{z_2^\prime}$ for some $z_1^\prime, z_2^\prime$ we have that $z_1^\prime = z_2^\prime$. Then $2_{z_1}=2_{z_2}$ and hence $z_1=z_2$, which contradicts $z_1 \neq z_2$. Thus, $\left(\left(k+1\right)_z\right)\alpha^{-1} \leqslant \left(k+1\right)_x.$
  Since $\left(\left(k+1\right)_z\right)\alpha^{-1}\notin\{ \mathbf{1}, 2_x, 3_x, \dots, k_x \},$ we have that $\left(\left(k+1\right)_z\right)\alpha^{-1} = \left(k+1\right)_x$, and hence $\left(\left(k+1\right)_x\right)\alpha = \left(k+1\right)_z$. We shall prove that $x=y$. The relation $2_x < \left(k+1\right)_x$ implies that $\left(2_x\right)\alpha < \left(\left(k+1\right)_x\right)\alpha$. Since $\left(2_x\right)\alpha = 2_y$ and $\left(\left(k+1\right)_x\right)\alpha = \left(k+1\right)_z$ we have that $2_y < \left(k+1\right)_z$, so $z=y$.
  \end{proof}

  For any $x\in \kappa$, consider the map $\pi_x \colon \sigma{\mathbb{N}^\kappa}\to \sigma{\mathbb{N}^\kappa}$ defined by the formula:
  \begin{equation*}
  \left(t\right)\left(a\right)\pi_x =
   \begin{cases}
     \left(t\right)a, &\text{if $t=x$};\\
     1, &\text{otherwise,}
   \end{cases}
  \end{equation*}
for any $a\in\sigma{\mathbb{N}^\kappa}$ and $t\in \kappa.$

  \begin{lemma}\label{lemma-1.2}
    Let $\kappa$ be any infinite cardinal and $\alpha \in H\left(\mathbb{I}\right)$ such that the equality $\left(2_x\right)\alpha = 2_x$ holds for any $x\in \kappa$. Then $\alpha$ is the identity map.
  \end{lemma}

  \begin{proof}[\textsl{Proof}]
  Let $a\in \sigma{\mathbb{N}^\kappa}$.
  Since the inequality $\left(a\right)\pi_x \leqslant a$ holds for any $x\in\kappa$ and $\alpha$ is an order isomorphism, it follows that $\left(\left(a\right)\pi_x\right)\alpha \leqslant \left(a\right)\alpha$.
  By Lemma~\ref{lemma-1.1} and by the lemma assumption we have that $\left(\left(a\right)\pi_x\right)\alpha=\left(a\right)\pi_x$, so $\left(a\right)\pi_x \leqslant \left(a\right)\alpha$ for any $x\in\kappa$ and therefore $a \leqslant \left(a\right)\alpha$.

  So, we have that $a \leqslant \left(a\right)\alpha$ for any $a \in \sigma{\mathbb{N}^\kappa}$ and for any $\alpha$ that satisfies the lemma assumption. Applying this result to the element $\left(a\right)\alpha$ and the map $\alpha^{-1}$ we have that $\left(a\right)\alpha \leqslant \left(\left(a\right)\alpha\right)\alpha^{-1} = a$.

  The inequalities $a \leqslant \left(a\right)\alpha$ and $\left(a\right)\alpha \leqslant a$ imply that $\left(a\right)\alpha=a$.
  \end{proof}

  \begin{theorem}\label{theorem-1.3}
  For any infinite cardinal $\kappa$, the group of units $H\left(\mathbb{I}\right)$ of the semigroup $\mathcal{IP\!F}\left(\sigma{\mathbb{N}^\kappa}\right)$ is isomorphic to the group $\mathcal{S}_\kappa$ of all bijections of the cardinal $\kappa$.
  Moreover $\alpha\in H\left(\mathbb{I}\right)$ if and only if $\alpha=\mathcal{F}^\circ_g$ for some $g\in\mathcal{S}_\kappa$.
  \end{theorem}

  \begin{proof}[\textsl{Proof}]
  Define the map $\mathcal{F}\colon \mathcal{S}_\kappa\to H\left(\mathbb{I}\right)$ in the following way:
  \begin{equation*}
  \forall g\in \mathcal{S}_\kappa \quad \left(g\right)\mathcal{F}=\mathcal{F}^\circ_g,
  \end{equation*}
  Since $\mathcal{F}^\circ_g$ is an order automorphism of the poset $\left(\sigma{\mathbb{N}^\kappa}, \leqslant\right)$ we have that the map $\mathcal{F}^\circ_g$ is an element of the group of units $H\left(\mathbb{I}\right)$, so $\mathcal{F}$ is well-defined.
  Next, we shall show that the map $\mathcal{F}$ is an isomorphism.

  Statement $\left(iv\right)$ of Lemma~\ref{lemma-1.2-3} implies that the map $\mathcal{F}$ is a homomorphism and statement $\left(vii\right)$ of Lemma~\ref{lemma-1.2-3} implies that $\mathcal{F}$ is injective.

  We shall show that $\mathcal{F}$ is surjective.
  Let $\alpha \in H\left(\mathbb{I}\right)$. Lemma~\ref{lemma-1.1} implies that for any $x\in \kappa$ there exists $y\in \kappa$ such that $\left(2_x\right)\alpha = 2_y$. We define the map $g\colon \kappa\to \kappa$ in the following way: $\left(x\right)g = y$. Since $\alpha$ is a bijection so is $g$.

  Now consider the composition $\alpha \circ \mathcal{F}^\circ_{g^{-1}}$.
  Let $x\in\kappa$. The definition of the map $g$ implies that
  \begin{equation*}
  \left(2_x\right)\left(\alpha \circ \mathcal{F}^\circ_{g^{-1}}\right)=\left(2_{\left(x\right)g}\right)\mathcal{F}^\circ_{g^{-1}}
  \end{equation*}
  and statement $\left(v\right)$ of Lemma~\ref{lemma-1.2-3} implies that $\left(2_{\left(x\right)g}\right)\mathcal{F}^\circ_{g^{-1}}=2_x$, so $\left(2_x\right)\left(\alpha \circ \mathcal{F}^\circ_{g^{-1}}\right)=2_x$. By Lemma~\ref{lemma-1.2}, $\alpha \circ \mathcal{F}^\circ_{g^{-1}}$ is identity map, therefore $\alpha=\left(\mathcal{F}^\circ_{g^{-1}}\right)^{-1}=\mathcal{F}^\circ_g.$
  \end{proof}

  Theorems~2.3 and 2.20 from~\cite{Clifford-Preston-1961-1967} and Theorem~\ref{theorem-1.3} imply the following corollary.

  \begin{corollary}\label{corollary-1.4}
    For any infinite cardinal $\kappa$ every maximal subgroup of the semigroup $\mathcal{IP\!F}\left(\sigma{\mathbb{N}^\kappa}\right)$ is isomorphic to the group $\mathcal{S}_\kappa$ of all bijections of the cardinal $\kappa$.
  \end{corollary}

  \begin{proposition}\label{proposition-1.5}
    For any infinite cardinal $\kappa$ and for any $\alpha \in \mathcal{IP\!F}\left(\sigma{\mathbb{N}^\kappa}\right)$ there exists a unique bijection $g_\alpha\in \mathcal{S}_\kappa$ such that $\alpha=\rho_{\alpha}\mathcal{F}^\circ_{g_\alpha}\lambda_{\alpha}$.
  \end{proposition}

  \begin{proof}[\textsl{Proof}]
  Let $\alpha \in \mathcal{IP\!F}\left(\sigma{\mathbb{N}^\kappa}\right)$. For the element $\rho^{-1}_{\alpha}\alpha\lambda^{-1}_{\alpha}$ we have that
  \begin{equation*}
  \rho_{\alpha}\rho^{-1}_{\alpha}\alpha\lambda^{-1}_{\alpha}\lambda_{\alpha} = \varepsilon\alpha\iota,
  \end{equation*}
  where $\varepsilon$ and $\iota$ are idempotents with $\operatorname{dom}\varepsilon = \operatorname{dom}\alpha$ and $\operatorname{dom}\iota = \operatorname{ran}\alpha$, so $\varepsilon\alpha\iota = \alpha$. Since
  \begin{equation*}
    \operatorname{dom}\left(\rho^{-1}_{\alpha}\alpha\lambda^{-1}_{\alpha}\right) = \operatorname{ran}\left({\rho^{-1}_{\alpha}\alpha\lambda^{-1}_{\alpha}}\right) = \sigma{\mathbb{N}^\kappa},
  \end{equation*}
   we have that ${\rho^{-1}_{\alpha}\alpha\lambda^{-1}_{\alpha}}\in H\left(\mathbb{I}\right)$. By Theorem~\ref{theorem-1.3}, for ${\rho^{-1}_{\alpha}\alpha\lambda^{-1}_{\alpha}}$ there exists a bijection ${g_\alpha}\in \mathcal{S}_\kappa$ such that ${\rho^{-1}_{\alpha}\alpha\lambda^{-1}_{\alpha}}=\mathcal{F}^\circ_{g_\alpha}$.

  Suppose that there exists $h\in \mathcal{S}_\kappa$ such that $\alpha = \rho_{\alpha}\mathcal{F}^\circ_h\lambda_{\alpha}$. Then the equality
  \begin{equation*}
  {\rho_{\alpha}\mathcal{F}^\circ_h\lambda_{\alpha} = \rho_{\alpha}\mathcal{F}^\circ_{g_\alpha}\lambda_{\alpha}}
  \end{equation*}
  implies that
  \begin{equation*}
    {\left(\rho^{-1}_{\alpha}\rho_{\alpha}\right)\mathcal{F}^\circ_h\left(\lambda_{\alpha}\lambda^{-1}_{\alpha}\right) = \left(\rho^{-1}_{\alpha}\rho_{\alpha}\right)\mathcal{F}^\circ_{g_\alpha}\left(\lambda_{\alpha}\lambda^{-1}_{\alpha}\right)}.
  \end{equation*}
  The definition of $\lambda_\alpha, \rho_\alpha$ implies that
  \begin{equation*}
    \rho^{-1}_{\alpha}\rho_{\alpha} = \lambda_{\alpha}\lambda^{-1}_{\alpha} = \mathbb{I},
  \end{equation*}
  so $\mathcal{F}^\circ_h = \mathcal{F}^\circ_{g_\alpha}$.
  Statement $\left(v\right)$ of Lemma~\ref{lemma-1.2-3} implies that $h={g_\alpha}$.
  \end{proof}

  The following corollary states that every order isomorphism $\alpha$ in the semigroup $\mathcal{IP\!F}\left(\sigma{\mathbb{N}^\kappa}\right)$ can be uniquely represented as a composition of three basic transformations: shifting to the origin of coordinates, an order isomorphism of entire $\sigma{\mathbb{N}^\kappa}$, and then shifting to the range of $\alpha$.

  \begin{corollary}\label{corollary-1.6}
    For any infinite cardinal $\kappa$ and for any element $\alpha \in \mathcal{IP\!F}\left(\sigma{\mathbb{N}^\kappa}\right)$ the representation $\alpha = \rho_{\alpha}\mathcal{F}^\circ_{g_\alpha}\lambda_{\alpha}$ is unique.
    \end{corollary}

    For any $\alpha\in \mathcal{IP\!F}\left(\sigma{\mathbb{N}^\kappa}\right)$ we shall use this notation $g_\alpha$ to denote the element of $S_\kappa$ that implements this representation $\alpha = \rho_{\alpha}\mathcal{F}^\circ_{g_\alpha}\lambda_{\alpha}$.

  \begin{lemma}\label{lemma-1.8}
    Let $\kappa$ be any infinite cardinal and $\alpha,\beta \in \mathcal{IP\!F}\left(\sigma{\mathbb{N}^\kappa}\right)$, then
  \begin{equation*}
    \begin{split}
      &d_{\alpha\beta}=\left(\operatorname{max}\{r_{\alpha},d_{\beta}\}-r_{\alpha}\right)\mathcal{F}_{g_\alpha}^{-1}+d_{\alpha};\\
      &r_{\alpha\beta}=\left(\operatorname{max}\{r_{\alpha},d_{\beta}\}-d_{\beta}\right)\mathcal{F}_{g_\beta}+r_{\beta};\\
      &\mathcal{F}^\circ_{g_{\alpha\beta}}=\mathcal{F}^\circ_{g_\alpha}\mathcal{F}^\circ_{g_\beta}.
    \end{split}
    \end{equation*}
  \end{lemma}

    \begin{proof}[\textsl{Proof}]
      By the definition of the composition of the partial maps:
      \begin{align*}
        \operatorname{dom}\left(\alpha\beta\right)&=\left(\operatorname{ran}\alpha\cap \operatorname{dom}\beta\right)\alpha^{-1}= \\
         &={\left({\uparrow}r_\alpha\cap {\uparrow}d_\beta\right)}\alpha^{-1}= \\
         &= \left({\uparrow}\operatorname{max}\{r_\alpha,d_\beta\}\right)\alpha^{-1}.
      \end{align*}
        Since $\alpha$ is an order isomorphism we get that
        \begin{equation*}
          {\left({\uparrow}{\operatorname{max}\{r_\alpha,d_\beta\}}\right)}\alpha^{-1} = {\uparrow}\left[{{\left(\operatorname{max}\{r_\alpha,d_\beta\}\right)}\alpha^{-1}}\right],
        \end{equation*}
        and then, by Corollary~\ref{corollary-1.6} and by Lemma~\ref{lemma-1.2-3}$\left[\left(vi\right),\left(viii\right)\right]$,
        \begin{equation*}
        \begin{split}
        \operatorname{dom}\left(\alpha\beta\right)&={\uparrow}\left[{{\left(\operatorname{max}\{r_\alpha,d_\beta\}\right)}\alpha^{-1}}\right]=\\
        &= {\uparrow}\left(\left[\operatorname{max}\{ r_\alpha,d_\beta \}\right]\lambda_\alpha^{-1}\left({\mathcal{F}^\circ_{g_\alpha}}\right)^{-1}\rho_\alpha^{-1}\right)=\\
        &= {\uparrow}\left(\left[\operatorname{max}\{ r_\alpha,d_\beta\}-r_{\alpha}+\mathbf{1}\right]\left({\mathcal{F}^\circ_{g_\alpha}}\right)^{-1}\rho_\alpha^{-1}\right)=\\
        &= {\uparrow}\left(\left[\left(\operatorname{max}\{ r_\alpha,d_\beta\}-r_{\alpha}\right)\mathcal{F}_{g_\alpha}^{-1}+\mathbf{1}\right]\rho_\alpha^{-1}\right)=\\
        &= {\uparrow}\left[\left(\operatorname{max}\{ r_\alpha,d_\beta\}-r_{\alpha}\right)\mathcal{F}_{g_\alpha}^{-1}+d_{\alpha}\right].
        \end{split}
        \end{equation*}

      Similarly, by the definition of the range of the composition of the partial maps:
      \begin{align*}
        \operatorname{ran}\left(\alpha\beta\right)&=\left(\operatorname{ran}\alpha\cap \operatorname{dom}\beta\right)\beta= \\
         & ={\left({\uparrow}r_\alpha\cap {\uparrow}d_\beta\right)}\beta =\\
         & = \left({\uparrow}\operatorname{max}\{r_\alpha,d_\beta\}\right)\beta.
      \end{align*}
        Since $\beta$ is an order isomorphism we get that
        \begin{equation*}
          {\left({\uparrow}{\operatorname{max}\{r_\alpha,d_\beta\}}\right)}\beta = {\uparrow}\left[{{\left(\operatorname{max}\{r_\alpha,d_\beta\}\right)}\beta}\right],
        \end{equation*}
        and then, by Corollary~\ref{corollary-1.6} and by Lemma~\ref{lemma-1.2-3}$\left[\left(vi\right),\left(viii\right)\right]$,
        \begin{equation*}
        \begin{split}
        \operatorname{ran}\left(\alpha\beta\right)&={\uparrow}\left[{{\left(\operatorname{max}\{r_\alpha,d_\beta\}\right)}\beta}\right]=\\
        &= {\uparrow}\left(\left[\operatorname{max}\{ r_\alpha,d_\beta \}\right]\lambda_\beta \mathcal{F}^\circ_{g_\beta}\rho_\beta\right)=\\
        &= {\uparrow}\left(\left[\operatorname{max}\{ r_\alpha,d_\beta\}-d_{\beta}+\mathbf{1}\right]\mathcal{F}^\circ_{g_\beta}\rho_\beta\right)=\\
        &= {\uparrow}\left(\left[\left(\operatorname{max}\{ r_\alpha,d_\beta\}-d_{\beta}\right)\mathcal{F}_{g_\beta}+\mathbf{1}\right]\rho_\beta\right)=\\
        &= {\uparrow}\left[\left(\operatorname{max}\{ r_\alpha,d_\beta\}-d_{\beta}\right)\mathcal{F}_{g_\beta}+r_{\beta}\right].
        \end{split}
        \end{equation*}

      We shall prove that
      \begin{equation*}
      \alpha\beta = \rho_{\alpha\beta}\mathcal{F}^\circ_{g_\alpha}\mathcal{F}^\circ_{g_\beta}\lambda_{\alpha\beta}.
      \end{equation*}
      The definition of the maps $\rho_{\alpha\beta}$, $\mathcal{F}^\circ_{g_\alpha}$, $\mathcal{F}^\circ_{g_\beta}$, $\lambda_{\alpha\beta}$ and the definition of the composition of the partial maps imply that
      \begin{equation*}
        \operatorname{dom}\left(\rho_{\alpha\beta}\mathcal{F}^\circ_{g_\alpha}\mathcal{F}^\circ_{g_\beta}\lambda_{\alpha\beta}\right) = \operatorname{dom}\left(\alpha\beta\right)
      \end{equation*}
        and
        \begin{equation*}
          \operatorname{ran}\left(\rho_{\alpha\beta}\mathcal{F}^\circ_{g_\alpha}\mathcal{F}^\circ_{g_\beta}\lambda_{\alpha\beta}\right)  = \operatorname{ran}\left(\alpha\beta\right).
        \end{equation*}
      Now consider any $a\in \operatorname{dom}\left(\alpha\beta\right)$ and the representation $a=d_{\alpha\beta}+a-d_{\alpha\beta}$. Denote $a-d_{\alpha\beta}$ by $b$, then $a$ has the representation $a = d_{\alpha\beta}+b$.
      And consider the images of $a$ under the maps $\alpha\beta$ and $\rho_{\alpha\beta}\mathcal{F}^\circ_{g_\alpha}\mathcal{F}^\circ_{g_\beta}\lambda_{\alpha\beta}$:
      \begin{align*}
        \left(a\right)\alpha\beta &= \left(d_{\alpha\beta}+b\right)\alpha\beta
        =\\&= \left(\left[\operatorname{max}\{ r_\alpha,d_\beta\}-r_{\alpha}\right]\mathcal{F}_{g_\alpha}^{-1}+d_{\alpha}+b\right)\alpha\beta
        =\\&= \left(\left[\operatorname{max}\{ r_\alpha,d_\beta\}-r_{\alpha}\right]\mathcal{F}_{g_\alpha}^{-1}+d_{\alpha}+b\right)\rho_{\alpha}\mathcal{F}^\circ_{g_\alpha}\lambda_{\alpha}\rho_{\beta}\mathcal{F}^\circ_{g_\beta}\lambda_{\beta}
        =\\&= \left(\left[\operatorname{max}\{ r_\alpha,d_\beta\}-r_{\alpha}\right]\mathcal{F}_{g_\alpha}^{-1}+\mathbf{1}+b\right)\mathcal{F}^\circ_{g_\alpha}\lambda_{\alpha}\rho_{\beta}\mathcal{F}^\circ_{g_\beta}\lambda_{\beta}
        =\\&= \left(\operatorname{max}\{ r_\alpha,d_\beta\}-r_{\alpha}+\mathbf{1}+\left(b\right)\mathcal{F}_{g_\alpha}\right)\lambda_{\alpha}\rho_{\beta}\mathcal{F}^\circ_{g_\beta}\lambda_{\beta}
        =\\&= \left(\operatorname{max}\{ r_\alpha,d_\beta\}+\left(b\right)\mathcal{F}_{g_\alpha}\right)\rho_{\beta}\mathcal{F}^\circ_{g_\beta}\lambda_{\beta}
        =\\&= \left(\operatorname{max}\{ r_\alpha,d_\beta\}-d_\beta+\mathbf{1}+\left(b\right)\mathcal{F}_{g_\alpha}\right)\mathcal{F}^\circ_{g_\beta}\lambda_{\beta}
        =\\&= \left(\left[\operatorname{max}\{ r_\alpha,d_\beta\}-d_\beta\right]\mathcal{F}_{g_\beta}+\mathbf{1}+\left(b\right)\mathcal{F}_{g_\alpha}\mathcal{F}_{g_\beta}\right)\lambda_{\beta}
        =\\&= \left[\operatorname{max}\{ r_\alpha,d_\beta\}-d_\beta\right]\mathcal{F}_{g_\beta}+\left(b\right)\mathcal{F}_{g_\alpha}\mathcal{F}_{g_\beta} + r_\beta
        =\\&= r_{\alpha\beta} + \left(b\right)\mathcal{F}_{g_\alpha}\mathcal{F}_{g_\beta};
        \end{align*}
      \begin{align*}
  \left(a\right)\rho_{\alpha\beta}\mathcal{F}^\circ_{g_\alpha}\mathcal{F}^\circ_{g_\beta}\lambda_{\alpha\beta}
   &= \left(d_{\alpha\beta}+b\right)\rho_{\alpha\beta}\mathcal{F}^\circ_{g_\alpha}\mathcal{F}^\circ_{g_\beta}\lambda_{\alpha\beta}=\\
   &= \left(b+\mathbf{1}\right)\mathcal{F}^\circ_{g_\alpha}\mathcal{F}^\circ_{g_\beta}\lambda_{\alpha\beta} =\\
   &= \left(\left(b\right)\mathcal{F}_{g_\alpha}\mathcal{F}_{g_\beta}+\mathbf{1}\right)\lambda_{\alpha\beta}=\\
   &=  \left(b\right)\mathcal{F}_{g_\alpha}\mathcal{F}_{g_\beta} + r_{\alpha\beta}.
\end{align*}
    We have that $\alpha\beta=\rho_{\alpha\beta}\mathcal{F}^\circ_{g_\alpha}\mathcal{F}^\circ_{g_\beta}\lambda_{\alpha\beta}$, so by Corollary~\ref{corollary-1.6} $\mathcal{F}^\circ_{g_{\alpha\beta}}=\mathcal{F}^\circ_{g_\alpha}\mathcal{F}^\circ_{g_\beta}$.
    \end{proof}

    \begin{corollary}\label{corollary-1.8c}
      For any infinite cardinal $\kappa$ and for any elements $\alpha, \beta \in \mathcal{IP\!F}\left(\sigma{\mathbb{N}^\kappa}\right)$ the bijection $g_{\alpha\beta}$ is equals to $g_\alpha g_\beta$.
    \end{corollary}

    \begin{corollary}\label{corollary-1.8d}
      Let $\kappa$ be any infinite cardinal and $\varepsilon$ be the idempotent of the semigroup $\mathcal{IP\!F}\left(\sigma{\mathbb{N}^\kappa}\right)$, then $g_\varepsilon = id_\kappa$, $\mathcal{F}^\circ_{g_{\varepsilon}}=\mathbb{I}$.
    \end{corollary}

  \begin{remark}\label{remark-1.9}
  In the bicyclic semigroup ${\mathscr{C}}\left(p,q\right)$ the semigroup operation is determined in the following way:
  \begin{equation*}
    p^iq^j\cdot p^kq^l=
  \left\{
    \begin{array}{ll}
      p^iq^{j-k+l}, & \hbox{if~} j>k;\\
      p^iq^l,       & \hbox{if~} j=k;\\
      p^{i-j+k}q^l, & \hbox{if~} j<k,
    \end{array}
  \right.
  \end{equation*}
  which is equivalent to the following formula:
  \begin{equation*}
    p^iq^j\cdot p^kq^l=p^{i+\operatorname{max}\{j,k\}-j}q^{l+\operatorname{max}\{j,k\}-k}.
  \end{equation*}
  We note that the bicyclic semigroup ${\mathscr{C}}\left(p,q\right)$ is isomorphic to the semigroup $\left(\mathbb{N}\times \mathbb{N},*\right)$ which is defined on the square $\mathbb{N}\times \mathbb{N}$ of the set of all positive integers with the following multiplication:
  \begin{equation}\label{eq-1.1}
    \left(i,j\right)*\left(k,l\right)=\left(i+\operatorname{max}\{j,k\}-j,l+\operatorname{max}\{j,k\}-k\right).
  \end{equation}
  To see this, it is sufficiently to check that the map
  \begin{equation*}
  f\colon {\mathscr{C}}\left(p,q\right) \rightarrow \mathbb{N}\times \mathbb{N} : p^iq^j \stackrel{f}{\mapsto} \left(i+1,j+1\right)
  \end{equation*}
  is an isomorphism between semigroups ${\mathscr{C}}\left(p,q\right)$ and $\left(\mathbb{N}\times \mathbb{N},*\right)$.
  \end{remark}

  In this paper we shall use the semigroup $\left(\mathbb{N}\times \mathbb{N},*\right)$ as a representation of the bicyclic semigroup ${\mathscr{C}}\left(p,q\right)$ and we shall denote the semigroup $\left(\mathbb{N}\times \mathbb{N},*\right)$ by $\mathbb{B}$.

  For any infinite cardinal $\kappa$, define the semigroup $\sigma{\mathbb{B}^\kappa}$ as the set $\sigma{\mathbb{N}^\kappa}\times \sigma{\mathbb{N}^\kappa}$ with the multiplications $*_\kappa$ which is similar to (\ref{eq-1.1}):
  \begin{equation}\label{eq-1.2}
    \left(a,b\right)*_\kappa\left(c,d\right)=\left(a+\operatorname{max}\{b,c\}-b,d+\operatorname{max}\{b,c\}-c\right), \hbox{ where } a, b, c, d\in \sigma{\mathbb{N}^\kappa}.
  \end{equation}

  We can observe that the semigroup $\sigma{\mathbb{B}^\kappa}$, as defined by the multiplication operation $*_\kappa$ in (\ref{eq-1.2}), is indeed isomorphic to the $\sigma$-product of $\kappa$ many copies of the bicyclic monoid.

  For any $g\in\mathcal{S}_\kappa$ consider a map $\Phi_g\colon \sigma{\mathbb{B}^\kappa} \to \sigma{\mathbb{B}^\kappa}$ defined in the following way:
  for any $\left(a,b\right)\in\sigma{\mathbb{B}^\kappa}$ define
  \begin{equation*}
  \left(\left(a,b\right)\right)\Phi_g = \left(\left(a\right)\mathcal{F}^\circ_g, \left(b\right)\mathcal{F}^\circ_g\right).
  \end{equation*}
  Statements $\left(i\right)$ and $\left(ii\right)$ of Lemma~\ref{lemma-1.2-3} imply that the map $\Phi_g$ is well-defined and $\Phi_g$ is a bijection.

  Check that the map $\Phi_g$ is an automorphism of $\sigma{\mathbb{B}^\kappa}$. For any $\left(a,b\right), \left(c,d\right)\in\sigma{\mathbb{B}^\kappa}$, by statements $\left(xiii\right)-\left(x\right)$ of Lemma~\ref{lemma-1.2-3}:
  \begin{equation*}
    \begin{split}
      \big(&\left(a,b\right)*_\kappa\left(c,d\right)\big)\Phi_g=\left(\left(a+\operatorname{max}\{b,c\}-b,d+\operatorname{max}\{b,c\}-c\right)\right)\Phi_g
      =\\&=\left(\left(a+\operatorname{max}\{b,c\}-b\right)\mathcal{F}^\circ_g,\left(d+\operatorname{max}\{b,c\}-c\right)\mathcal{F}^\circ_g\right)
      =\\&=\left(\left(a\right)\mathcal{F}_g+\operatorname{max}\{\left(b\right)\mathcal{F}_g,\left(c\right)\mathcal{F}_g\}-\left(b\right)\mathcal{F}_g,\left(d\right)\mathcal{F}_g+\operatorname{max}\{\left(b\right)\mathcal{F}_g,\left(c\right)\mathcal{F}_g\}-\left(c\right)\mathcal{F}_g\right)
      =\\&=\left(\left(a\right)\mathcal{F}_g,\left(b\right)\mathcal{F}_g\right)*_\kappa\left(\left(c\right)\mathcal{F}_g,\left(d\right)\mathcal{F}_g\right)=\left(\left(a\right)\mathcal{F}^\circ_g,\left(b\right)\mathcal{F}^\circ_g\right)*_\kappa\left(\left(c\right)\mathcal{F}^\circ_g,\left(d\right)\mathcal{F}^\circ_g\right)
      =\\&=\left(a,b\right)\Phi_g*_\kappa\left(c,d\right)\Phi_g.
    \end{split}
  \end{equation*}

  Let $\kappa$ be any infinite cardinal and $\operatorname{\mathbf{Aut}}(\sigma{\mathbb{B}^\kappa})$ be the group of automorphisms of the semigroup $\sigma{\mathbb{B}^\kappa}$.
  Consider the map $\Phi \colon \mathcal{S}_\kappa \to \operatorname{\mathbf{Aut}}\left({\sigma{\mathbb{B}^\kappa}}\right)$
  for any $g\in\mathcal{S}_\kappa$ define $\left(g\right)\Phi=\Phi_g$.
  Statement $\left(vii\right)$ of Lemma~\ref{lemma-1.2-3} implies that $\Phi$ is injective. Next, we show that the map $\Phi$ is a homomorphism.
  For any $g,h\in\mathcal{S}_\kappa$ consider the image of their composition: for any $\left[a,b\right]\in\sigma{\mathbb{B}^\kappa}$
  \begin{equation*}
    \begin{split}
      \left(\left[a,b\right]\right)\left(gh\right)\Phi&=\left(\left[a,b\right]\right)\Phi_{gh}=\\
      &=\left[\left(a\right)\mathcal{F}^\circ_{gh}, \left(b\right)\mathcal{F}^\circ_{gh}\right].
    \end{split}
  \end{equation*}
  Statement $\left(iv\right)$ of Lemma~\ref{lemma-1.2-3} implies that
  \begin{equation*}
    \left[\left(a\right)\mathcal{F}^\circ_{gh}, \left(b\right)\mathcal{F}^\circ_{gh}\right] = \left[\left(a\right)\mathcal{F}^\circ_g\mathcal{F}^\circ_h, \left(b\right)\mathcal{F}^\circ_g\mathcal{F}^\circ_h\right],
  \end{equation*}
  and since
  \begin{align*}
    \left[\left(a\right)\mathcal{F}^\circ_g\mathcal{F}^\circ_h, \left(b\right)\mathcal{F}^\circ_g\mathcal{F}^\circ_h\right]&=\left(\left[\left(a\right)\mathcal{F}^\circ_g, \left(b\right)\mathcal{F}^\circ_g\right]\right)\Phi_h= \\
      & =\left(\left[a, b\right]\right)\Phi_g\Phi_h =\\
      & =\left(\left[a, b\right]\right)\left(g\right)\Phi\left(h\right)\Phi,
  \end{align*}
  we have that
  \begin{equation*}
    \left(\left[a,b\right]\right)\left(gh\right)\Phi=\left(\left[a, b\right]\right)\left(g\right)\Phi\left(h\right)\Phi,
  \end{equation*}
  i.e., $\Phi$ is a homomorphism.

  For any infinite cardinal $\kappa$ consider the semidirect product $\mathcal{S}_\kappa\ltimes_{\Phi}\sigma{\mathbb{B}^\kappa}$ of the semigroup $\sigma{\mathbb{B}^\kappa}$ by the group  $\mathcal{S}_\kappa$ as the set $\mathcal{S}_\kappa\times \sigma{\mathbb{B}^\kappa}$ with the operation:
  \begin{equation*}
    \left(g, \left[a,b\right]\right)\left(h, \left[c,d\right]\right)=\left(gh, \left(\left[a,b\right]\right)\Phi_h*_\kappa\left[c,d\right]\right)  \quad \hbox{ for } \left(g, \left[a,b\right]\right),\left(h, \left[c,d\right]\right) \in \mathcal{S}_\kappa\times \sigma{\mathbb{B}^\kappa}.
    \end{equation*}
    Define the map $\Psi\colon{} \mathcal{IP\!F}\left(\sigma{\mathbb{N}^\kappa}\right)\to\mathcal{S}_\kappa\ltimes_{\Phi}\sigma{\mathbb{B}^\kappa}$ by the formula:
    \begin{equation*}
      \left(\alpha\right)\Psi=\left(g_\alpha,\left[\left(d_\alpha\right)\mathcal{F}^\circ_{g_\alpha}, r_\alpha\right]\right).
    \end{equation*}
    The definition of $d_\alpha, r_\alpha, g_\alpha$ and $\mathcal{F}^\circ_{g_\alpha}$ implies that the map $\Psi$ is well-defined.

  \begin{theorem}\label{theorem-1.10}
    For any infinite cardinal $\kappa$ the semigroup $\mathcal{IP\!F}\left(\sigma{\mathbb{N}^\kappa}\right)$ is isomorphic to the semidirect product $\mathcal{S}_\kappa\ltimes_{\Phi}
    \sigma{\mathbb{B}^\kappa}$ of the semigroup $\sigma{\mathbb{B}^\kappa}$ by the group  $\mathcal{S}_\kappa$.
    \end{theorem}

    \begin{proof}[\textsl{Proof}]
    Consider the map $\Psi$. Corollary~\ref{corollary-1.6} implies that $\Psi$ is a bijection. We shall prove that $\Psi$ is also a homomorphism.

    For any $\alpha, \beta \in \mathcal{IP\!F}\left(\sigma{\mathbb{N}^\kappa}\right)$ we have that
    $\left(\alpha\beta\right)\Psi = \left(g_{\alpha\beta}, \left[\left(d_{\alpha\beta}\right)\mathcal{F}^\circ_{g_{\alpha\beta}}, r_{\alpha\beta}\right]\right).$
    Co\-rol\-lary~\ref{corollary-1.8c} and Lemma~\ref{lemma-1.8} imply that
    \begin{align*}
      &\left(g_{\alpha\beta}, \left[\left(d_{\alpha\beta}\right)\mathcal{F}^\circ_{g_{\alpha\beta}}, r_{\alpha\beta}\right]\right)= \\
       &=
    \left(g_\alpha g_\beta,
    \left[
      \left(\left(\operatorname{max}\{r_{\alpha},d_{\beta}\}-r_{\alpha}\right)\mathcal{F}_{g_\alpha}^{-1}+d_{\alpha}\right)\mathcal{F}^\circ_{g_\alpha}\mathcal{F}^\circ_{g_\beta},
      \left(\operatorname{max}\{r_{\alpha},d_{\beta}\}-d_{\beta}\right)\mathcal{F}_{g_\beta}+r_{\beta}
    \right]\right).
    \end{align*}
    Lemma~\ref{lemma-1.2-3}, the definition of the operation $*_\kappa$, and the definition of the map $\Phi$ imply that
    \begin{equation*}
      \begin{split}
        &\left(g_\alpha g_\beta,
        \left[
          \left(\left(\operatorname{max}\{r_{\alpha},
          d_{\beta}\}-r_{\alpha}\right)\mathcal{F}_{g_\alpha}^{-1}+d_{\alpha}\right)\mathcal{F}^\circ_{g_\alpha}\mathcal{F}^\circ_{g_\beta},
          \left(\operatorname{max}\{r_{\alpha},
            d_{\beta}\}-d_{\beta}\right)\mathcal{F}_{g_\beta}+r_{\beta}
        \right]\right)=
        \\ &=\big(g_\alpha g_\beta,
        \big[
         \operatorname{max}\{\left(r_{\alpha}\right)\mathcal{F}_{g_\beta},
         \left(d_{\beta}\right)\mathcal{F}_{g_\beta}\}-\left(r_{\alpha}\right)\mathcal{F}_{g_\beta}+
         \left(d_{\alpha}\right)\mathcal{F}_{g_\alpha}\mathcal{F}_{g_\beta},
          \operatorname{max}\{\left(r_{\alpha}\right)\mathcal{F}_{g_\beta}, \\
          & \qquad \qquad \left(d_{\beta}\right)\mathcal{F}_{g_\beta}\}-\left(d_{\beta}\right)\mathcal{F}_{g_\beta}+r_{\beta}
        \big]\big)=
        \\&=\left(g_\alpha g_\beta,
        \left[\left(d_{\alpha}\right)\mathcal{F}_{g_\alpha}\mathcal{F}_{g_\beta}, \left(r_{\alpha}\right)\mathcal{F}_{g_\beta}\right]
        *_\kappa
        \left[\left(d_{\beta}\right)\mathcal{F}_{g_\beta}, r_\beta\right]
        \right)=\\
        &=\left(g_\alpha g_\beta,
        \left[\left(d_{\alpha}\right)\mathcal{F}^\circ_{g_\alpha}\mathcal{F}^\circ_{g_\beta}, \left(r_{\alpha}\right)\mathcal{F}^\circ_{g_\beta}\right]
        *_\kappa
        \left[\left(d_{\beta}\right)\mathcal{F}^\circ_{g_\beta}, r_\beta\right]
        \right)=
        \\&=\left(g_\alpha g_\beta,
        \left(\left[\left(d_{\alpha}\right)\mathcal{F}^\circ_{g_\alpha}, r_{\alpha}\right]\right)\Phi_{g_\beta}
        *_\kappa
        \left[\left(d_{\beta}\right)\mathcal{F}^\circ_{g_\beta}, r_\beta\right]
        \right)=\\
        &=
        \left(g_\alpha, \left[\left(d_{\alpha}\right)\mathcal{F}^\circ_{g_\alpha}, r_\alpha\right]\right)\left(g_\beta, \left[\left(d_{\beta}\right)\mathcal{F}^\circ_{g_\beta}, r_\beta\right]\right)\\
        &=\left(\alpha\right)\Psi\left(\beta\right)\Psi.
      \end{split}
    \end{equation*}
    \end{proof}

    For any $\alpha\in\mathcal{IP\!F}\left(\sigma{\mathbb{N}^\kappa}\right)$, let $\left(g_\alpha,\left[\left(d_\alpha\right)\mathcal{F}^\circ_{g_\alpha}, r_\alpha \right]\right)= \left(\alpha\right)\Psi$ be the image of the element $\alpha$ by the isomorphism $\Psi\colon \mathcal{IP\!F}\left(\sigma{\mathbb{N}^\kappa}\right) \to \mathcal{S}_\kappa\ltimes_{\Phi}\sigma{\mathbb{B}^\kappa}$ which is defined above the proof of Theorem~\ref{theorem-1.10}.

    Every inverse semigroup $S$ admits the \emph{least group} congruence $\mathfrak{C}_{\mathbf{mg}}$ (see \cite[Section~III]{Petrich-1984}):
    \begin{equation*}
        s\mathfrak{C}_{\mathbf{mg}} t \quad \hbox{if and only if \quad there exists an idempotent} \quad e\in S \quad \hbox{such that} \quad se=te.
    \end{equation*}

    \begin{proposition}\label{proposition-1.11}
      For any infinite cardinal $\kappa$, any element $\alpha \in \mathcal{IP\!F}\left(\sigma{\mathbb{N}^\kappa}\right)$ and for any idempotent $\varepsilon \in \mathcal{IP\!F}\left(\sigma{\mathbb{N}^\kappa}\right)$ we have:
      \begin{equation*}
        \begin{split}
        \left(\alpha\varepsilon\right)\Psi&=
        \left(g_\alpha,\left[\left(d_\alpha\right)\mathcal{F}^\circ_{g_\alpha}, r_\alpha \right]\right) \left(id_\kappa, \left[d_\varepsilon, d_\varepsilon\right]\right)=\\
        &=\left(g_\alpha,\left[\operatorname{max}\{r_\alpha,d_\varepsilon\}-r_\alpha+\left(d_\alpha\right)\mathcal{F}^\circ_{g_\alpha}, \operatorname{max}\{r_\alpha,d_\varepsilon\} \right]\right);\\
        \left(\varepsilon\alpha\right)\Psi&=
        \left(id_\kappa, \left[d_\varepsilon, d_\varepsilon\right]\right) \left(g_\alpha,\left[\left(d_\alpha\right)\mathcal{F}^\circ_{g_\alpha}, r_\alpha \right]\right) =\\
        &=\left(g_\alpha,\left[\left(\operatorname{max}\{d_\varepsilon, d_\alpha\}\right)\mathcal{F}^\circ_{g_\alpha}, \left(\operatorname{max}\{d_\varepsilon, d_\alpha\}\right)\mathcal{F}^\circ_{g_\alpha}-\left(d_\alpha\right)\mathcal{F}^\circ_{g_\alpha}+r_\alpha \right]\right).
      \end{split}
      \end{equation*}
      \end{proposition}

      \begin{proof}[\textsl{Proof}]
        By Corollary~\ref{corollary-1.8d}, $g_\varepsilon$ is the identity permutation, i.e., $g_\varepsilon=id_\kappa$ and $\mathcal{F}^\circ_{g_\varepsilon}=\mathbb{I}$. Since $\operatorname{dom}\varepsilon=\operatorname{ran}\varepsilon$ we have that $d_\varepsilon = r_\varepsilon$ and then $\left(d_\varepsilon\right)\mathcal{F}^\circ_{g_\varepsilon}=d_\varepsilon = r_\varepsilon$, so
        \begin{equation*}
        \left(g_\varepsilon,\left[\left(d_\varepsilon\right)\mathcal{F}^\circ_{g_\varepsilon}, r_\varepsilon\right]\right) = \left(id_\kappa, \left[d_\varepsilon, d_\varepsilon\right]\right).
        \end{equation*}
        Then the definition of the multiplication in $\mathcal{S}_\kappa\ltimes_{\Phi}\sigma{\mathbb{B}^\kappa}$ completes the proof of the proposition.
      \end{proof}

    The following theorem describes the least group congruence on the semigroup $\mathcal{IP\!F}\left(\sigma{\mathbb{N}^\kappa}\right)$.

    \begin{theorem}\label{theorem-1.11}
    Let $\kappa$ be any infinite cardinal. Then $\alpha\mathfrak{C}_{\mathbf{mg}} \beta$ in the semigroup $\mathcal{IP\!F}\left(\sigma{\mathbb{N}^\kappa}\right)$ if and only if
    \begin{equation*}
      g_\alpha=g_\beta \qquad \hbox{and} \qquad \left(d_\alpha\right)\mathcal{F}^\circ_{g_\alpha}-r_\alpha=\left(d_\beta\right)\mathcal{F}^\circ_{g_\beta}-r_\beta.
    \end{equation*}
    \end{theorem}

    \begin{proof}[\textsl{Proof}]
    Fix an idempotent $\varepsilon$ in $\mathcal{IP\!F}\left(\sigma{\mathbb{N}^\kappa}\right)$. By Proposition~\ref{proposition-1.11},
    \begin{equation*}
    \begin{split}
      &\left(g_\alpha,\left[\left(d_\alpha\right)\mathcal{F}^\circ_{g_\alpha}, r_\alpha \right]\right) \left(id_\kappa, \left[d_\varepsilon, d_\varepsilon\right]\right)=
        \left(g_\alpha,\left[\operatorname{max}\{r_\alpha,d_\varepsilon\}-r_\alpha+\left(d_\alpha\right)\mathcal{F}^\circ_{g_\alpha}, \operatorname{max}\{r_\alpha,d_\varepsilon\} \right]\right),\\
      &\left(g_\beta,\left[\left(d_\beta\right)\mathcal{F}^\circ_{g_\beta}, r_\beta \right]\right) \left(id_\kappa, \left[d_\varepsilon, d_\varepsilon\right]\right)=
        \left(g_\beta,\left[\operatorname{max}\{r_\beta,d_\varepsilon\}-r_\beta+\left(d_\beta\right)\mathcal{F}^\circ_{g_\beta}, \operatorname{max}\{r_\beta,d_\varepsilon\} \right]\right),
    \end{split}
    \end{equation*}
    so the equality $\alpha\varepsilon=\beta\varepsilon$ holds if and only if
    \begin{equation*}
      g_\alpha=g_\beta \qquad \hbox{and} \qquad \left(d_\alpha\right)\mathcal{F}^\circ_{g_\alpha}-r_\alpha=\left(d_\beta\right)\mathcal{F}^\circ_{g_\beta}-r_\beta.
    \end{equation*}
    \end{proof}

    For any infinite cardinal $\kappa$, by $\sigma{\mathbb{Z}^\kappa_+}$ we shall denote the group $\left(\sigma{\mathbb{Z}^\kappa}, +\right)$.
    Let $\operatorname{\mathbf{Aut}}(\sigma{\mathbb{Z}^\kappa_+})$ be the group of automorphisms of the group $\sigma{\mathbb{Z}^\kappa_+}$.
    Consider the map $\Theta \colon \mathcal{S}_\kappa \to \operatorname{\mathbf{Aut}}\left(\sigma{\mathbb{Z}^\kappa_+}\right)$:
    for any $g\in\mathcal{S}_\kappa$ define $\left(g\right)\Theta=\mathcal{F}^\diamond_g$.

    Statements $\left(i\right), \left(iii\right)$ and $\left(viii\right)$ of Lemma~\ref{lemma-1.2-3} imply that for any $g\in \mathcal{S}$ the map $\mathcal{F}^\diamond_g$ is an isomorphism of the group $\sigma{\mathbb{Z}^\kappa_+}$, so the map $\Theta$ is well-defined.
    Next, statements $\left(iv\right)$ and $\left(vii\right)$ of Lemma~\ref{lemma-1.2-3} imply that
    the map $\Theta$ is an injective homomorphism.

    Consider the semidirect product $\mathcal{S}_\kappa\ltimes_{\Theta}\left(\sigma{\mathbb{Z}^\kappa}, +\right)$ as the set $\mathcal{S}_\kappa\times \sigma{\mathbb{Z}^\kappa}$ with the operation
    \begin{equation*}
    \left(g, m\right)\left(h, n\right)=\left(gh, \left(m\right)\mathcal{F}^\diamond_h+n\right).
    \end{equation*}

    \begin{theorem}\label{theorem-1.12}
      For any infinite cardinal $\kappa$ the quotient semigroup $\mathcal{IP\!F}\left(\sigma{\mathbb{N}^\kappa}\right)/\mathfrak{C}_{\mathbf{mg}}$ is isomorphic to the semidirect product $\mathcal{S}_\kappa\ltimes_{\Theta} \left(\sigma{\mathbb{Z}^\kappa}, +\right)$ of the group $\left(\sigma{\mathbb{Z}^\kappa}, +\right)$ by the group $\mathcal{S}_\kappa$.
    \end{theorem}

    \begin{proof}[\textsl{Proof}]
      Define the map $\Upsilon\colon\mathcal{IP\!F}\left(\sigma{\mathbb{N}^\kappa}\right)\to \mathscr{S}_\kappa\ltimes_{\Theta}\left(\sigma{\mathbb{Z}^\kappa}, +\right)$ in the following way: for any $\alpha \in \mathcal{IP\!F}\left(\sigma{\mathbb{N}^\kappa}\right)$ we put
      \begin{equation*}
      \left(\alpha\right)\Upsilon=\left(g_\alpha, \left(d_\alpha\right)\mathcal{F}^\circ_{g_\alpha} - r_\alpha\right).
      \end{equation*}
      Since $a-b\in \sigma{\mathbb{Z}^\kappa}$ for any $a, b \in \sigma{\mathbb{N}^\kappa}$ we have that $\Upsilon$ is well-defined.

      For any $\alpha,\beta \in \mathcal{IP\!F}\left(\sigma{\mathbb{N}^\kappa}\right)$ by the definition of $\Upsilon$ we have that
    \begin{equation*}
        \left(\alpha\beta\right)\Upsilon = \left(g_{\alpha\beta}, \left(d_{\alpha\beta}\right)\mathcal{F}^\circ_{g_{\alpha\beta}} - r_{\alpha\beta} \right),
    \end{equation*}
    and by Lemma~\ref{lemma-1.8}
    \begin{equation*}
      \left(\alpha\beta\right)\Upsilon
      {=}\left(g_\alpha g_\beta, \left(\left(\operatorname{max}\{ r_\alpha, d_\beta \}{-}r_\alpha\right) \mathcal{F}_{g_\alpha}^{-1} + d_\alpha\right) \mathcal{F}^\circ_{g_\alpha} \mathcal{F}^\circ_{g_\beta} {-} \left(\operatorname{max}\{ r_\alpha, d_\beta \} {-}d_\beta\right) \mathcal{F}_{g_\beta} {-} r_\beta \right),
    \end{equation*}
    then, by statements $\left(viii\right)$ and $\left(ix\right)$ of Lemma~\ref{lemma-1.2-3}
    \begin{equation*}
      \begin{split}
      \left(\alpha\beta\right)\Upsilon
      =&\big(g_\alpha g_\beta, \left(\operatorname{max}\{ r_\alpha, d_\beta \}\right)\mathcal{F}_{g_\beta} - \left(r_\alpha\right)\mathcal{F}_{g_\beta} + \left(d_\alpha\right)\mathcal{F}_{g_\alpha}\mathcal{F}_{g_\beta} - \left(\operatorname{max}\{ r_\alpha, d_\beta \}\right)\mathcal{F}_{g_\beta} +\\
       &\qquad \qquad + \left(d_\beta\right)\mathcal{F}_{g_\beta}  - r_\beta \big)=\\
      =&\left(g_\alpha g_\beta, \left(d_\alpha\right)\mathcal{F}_{g_\alpha}\mathcal{F}_{g_\beta} - \left(r_\alpha\right)\mathcal{F}_{g_\beta} + \left(d_\beta\right)\mathcal{F}_{g_\beta} - r_\beta \right)=\\
      =&\left(g_\alpha, \left(d_\alpha\right)\mathcal{F}^\circ_{g_\alpha} - r_\alpha\right) \left(g_\beta, \left(d_\beta\right)\mathcal{F}^\circ_{g_\beta} - r_\beta\right)=\\
      =&\left(\alpha\right)\Upsilon\left(\beta\right)\Upsilon,
      \end{split}
      \end{equation*}
      and hence $\Upsilon$ is a  homomorphism.

      Show that the map $\Upsilon$ is surjective. For any $\left(g, z\right) \in \mathcal{S}_\kappa\times\sigma{\mathbb{Z}^\kappa}$, consider the maps $a,b\colon \kappa\to \mathbb{N}$. For any $x\in \kappa$:
      \begin{equation*}
        \left(x\right)a=
         \begin{cases}
           \left(x\right)z, &\text{if } \left(x\right)z > 0\\
           1, &\text{if } \left(x\right)z = 0\\
           0, &\text{if } \left(x\right)z < 0
         \end{cases}
        \qquad \text{and} \qquad
        \left(x\right)b=
         \begin{cases}
           0, &\text{if } \left(x\right)z > 0\\
           1, &\text{if } \left(x\right)z = 0\\
           -\left(x\right)z, &\text{if } \left(x\right)z < 0.
         \end{cases}
      \end{equation*}
    We have that $a,b\in\sigma{\mathbb{N}^\kappa}$ and $z=a-b$. Now we consider $\alpha\in\mathcal{IP\!F}\left(\sigma{\mathbb{N}^\kappa}\right)$ such that
    \begin{equation*}
      \begin{split}
        g_\alpha&=g,\\
        d_\alpha&=\left(a\right)\left(\mathcal{F}^\circ_g\right)^{-1},\\
        r_\alpha&=b.
      \end{split}
    \end{equation*}
    Then
    \begin{align*}
      \left(\alpha\right)\Upsilon& = \left(g_\alpha, \left(d_\alpha\right)\mathcal{F}^\circ_{g_\alpha} - r_\alpha\right)= \\
       &= \left(g, \left(\left(a\right)\left(\mathcal{F}^\circ_g\right)^{-1}\right)\mathcal{F}^\circ_g - b\right)=\\
       & = \left(g, a - b\right) =\\
       &= \left(g, z\right),
    \end{align*}
     so $\Upsilon$ is surjective.

    Also, Theorem~\ref{theorem-1.11} implies that $\alpha\mathfrak{C}_{\mathbf{mg}} \beta$ in $\mathcal{IP\!F}\left(\sigma{\mathbb{N}^\kappa}\right)$ if and only if $\left(\alpha\right)\Upsilon=\left(\beta\right)\Upsilon$. This implies that the homomorphism $\Upsilon$ generates the congruences $\mathfrak{C}_{\mathbf{mg}}$ on $\mathcal{IP\!F}\left(\sigma{\mathbb{N}^\kappa}\right)$.
    \end{proof}

    Every inverse semigroup $S$ admits a partial order:
    \begin{equation*}
      a\preccurlyeq b \qquad \hbox{if and only if there exists} \qquad e\in E\left(S\right) \quad \hbox{such that} \quad a=b e.
    \end{equation*}
    So defined order is called \emph{the natural partial order} on $S$. We observe that $a\preccurlyeq b$ in an inverse semigroup $S$ if and only if $a=f b$ for some $f\in E\left(S\right)$ (see \cite[Lemma~1.4.6]{Lawson-1998}).

    This and Proposition \ref{proposition-1.11} imply the following proposition, which describes the natural partial order on the semigroup $\mathcal{IP\!F}\left(\sigma{\mathbb{N}^\kappa}\right)$.

    \begin{proposition}\label{proposition-2.22}
      Let $\kappa$ be any infinite cardinal and let $\alpha,\beta\in\mathcal{IP\!F}\left(\sigma{\mathbb{N}^\kappa}\right)$. Then the following conditions are equivalent:
    \begin{itemize}
      \item[$\left(i\right)$] $\alpha\preccurlyeq\beta$;
      \item[$\left(ii\right)$] $g_\alpha=g_\beta$, $\left(d_\alpha\right)\mathcal{F}^\circ_{g_\alpha} - r_\alpha=\left(d_\beta\right)\mathcal{F}^\circ_{g_\beta} - r_\beta$ and $d_\beta\leqslant d_\alpha$ in the poset $\left(\sigma{\mathbb{N}^\kappa},\leqslant\right)$;
      \item[$\left(iii\right)$] $g_\alpha=g_\beta$, $\left(d_\alpha\right)\mathcal{F}^\circ_{g_\alpha} - r_\alpha=\left(d_\beta\right)\mathcal{F}^\circ_{g_\beta} - r_\beta$ and $r_\beta\leqslant r_\alpha$ in the poset $\left(\sigma{\mathbb{N}^\kappa},\leqslant\right)$.
    \end{itemize}
    \end{proposition}

    An inverse semigroup $S$ is said to be \emph{$E$-unitary} if $a e\in E\left(S\right)$ for some $e\in E\left(S\right)$ implies that $a\in E\left(S\right)$ \cite{Lawson-1998}. $E$-unitary inverse semigroups were introduced by Siat\^{o} in \cite{Saito-1965}, where they were called ``\emph{proper ordered inverse semigroups}''.

    \begin{proposition}\label{proposition-2.23}
      For any infinite cardinal $\kappa$, the inverse semigroup $\mathcal{IP\!F}\left(\sigma{\mathbb{N}^\kappa}\right)$ is \linebreak $E$-unitary.
    \end{proposition}

    \begin{proof}[\textsl{Proof}]
      Let $\alpha\in\mathcal{IP\!F}\left(\sigma{\mathbb{N}^\kappa}\right)$. Suppose that $\alpha\varepsilon$ is an idempotent for some idempotent $\varepsilon\in\mathcal{IP\!F}\left(\sigma{\mathbb{N}^\kappa}\right)$. Then Proposition \ref{proposition-1.11} and the definition of idempotents imply that $g_\alpha=id_\kappa$ and $d_\alpha=\left(d_\alpha\right)\mathcal{F}_{g_\alpha}=r_\alpha$, so $\alpha$ is an idempotent.
    \end{proof}

    An inverse semigroup $S$ is called \emph{$F$-inverse}, if the $\mathfrak{C}_{\textsf{mg}}$-class $s_{\mathfrak{C}_{\textsf{mg}}}$ of each element $s$ has the top (biggest) element with the respect to the natural partial order on $S$ \cite{McFadden-Carroll-1971}.

    \begin{proposition}\label{proposition-2.24}
    For any infinite cardinal $\kappa$, the semigroup $\mathcal{IP\!F}\left(\sigma{\mathbb{N}^\kappa}\right)$ is an $F$-inverse semigroup.
    \end{proposition}

    \begin{proof}[\textsl{Proof}]
    Let $\alpha\in\mathcal{IP\!F}\left(\sigma{\mathbb{N}^\kappa}\right)$. Consider an element $\beta \in\mathcal{IP\!F}\left(\sigma{\mathbb{N}^\kappa}\right)$ such that
    \begin{equation*}
      \begin{split}
        g_\beta&=g_\alpha,\\
        d_\beta&=d_\alpha - \operatorname{min}\{d_\alpha, \left(r_\alpha\right)\left(\mathcal{F}^\circ_{g_\alpha}\right)^{-1} \} + \mathbf{1},\\
        r_\beta&=r_\alpha - \operatorname{min}\{\left(d_\alpha\right)\mathcal{F}^\circ_{g_\alpha}, r_\alpha \} + \mathbf{1}.
      \end{split}
    \end{equation*}
    We have that $\operatorname{min}\{d_\alpha, \left(r_\alpha\right)\left(\mathcal{F}^\circ_{g_\alpha}\right)^{-1} \} \in \sigma{\mathbb{N}^\kappa}$ and $\operatorname{min}\{d_\alpha, \left(r_\alpha\right)\left(\mathcal{F}^\circ_{g_\alpha}\right)^{-1} \} \leqslant d_\alpha$, so $d_\beta \in \sigma{\mathbb{N}^\kappa}$. Similar $r_\beta \in \sigma{\mathbb{N}^\kappa}$, so $\beta$ is well-defined.
    Also, we have that $g_\beta=g_\alpha$ and
    \begin{equation*}
      \begin{split}
        \left(d_\beta\right)\mathcal{F}^\circ_{g_\beta}{-}r_\beta=&
        \left(d_\alpha {-} \operatorname{min}\{d_\alpha, \left(r_\alpha\right)\left(\mathcal{F}^\circ_{g_\alpha}\right)^{-1} \} {+} \mathbf{1}\right)\mathcal{F}^\circ_{g_\alpha}{-}
        \left(r_\alpha {-} \operatorname{min}\{\left(d_\alpha\right)\mathcal{F}^\circ_{g_\alpha}, r_\alpha \} {+} \mathbf{1}\right)=\\
        =&\left(d_\alpha\right)\mathcal{F}^\circ_{g_\alpha} - \operatorname{min}\{\left(d_\alpha\right)\mathcal{F}^\circ_{g_\alpha}, r_\alpha \} + \mathbf{1}
        -r_\alpha + \operatorname{min}\{\left(d_\alpha\right)\mathcal{F}^\circ_{g_\alpha}, r_\alpha \} - \mathbf{1}=\\
        =&\left(d_\alpha\right)\mathcal{F}^\circ_{g_\alpha} - r_\alpha,
      \end{split}
    \end{equation*}
    then Theorem~\ref{theorem-1.11} implies that $\beta\mathfrak{C}_{\mathbf{mg}}\alpha$.

    Now, for any $\gamma\in\mathcal{IP\!F}\left(\sigma{\mathbb{N}^\kappa}\right)$, such that $\gamma\mathfrak{C}_{\mathbf{mg}}\alpha$, we consider the idempotent $\varepsilon$ with $d_\varepsilon=r_\gamma$ and consider the product $\left(\beta\right)\Psi\left(\varepsilon\right)\Psi$. By Proposition~\ref{proposition-1.11}
    \begin{equation*}
      \begin{split}
        \left(\beta\right)\Psi\left(\varepsilon\right)\Psi &= \left(g_\beta,\left[\left(d_\beta\right)\mathcal{F}^\circ_{g_\beta}, r_\beta \right]\right) \left(id_\kappa, \left[d_\varepsilon, d_\varepsilon\right]\right)=\\
        &=
        \left(g_\beta,\left[\operatorname{max}\{r_\beta,d_\varepsilon\}-r_\beta+\left(d_\beta\right)\mathcal{F}^\circ_{g_\beta}, \operatorname{max}\{r_\beta,d_\varepsilon\} \right]\right) =\\
        &=\left(g_\beta,\left[\operatorname{max}\{r_\beta,r_\gamma\}-r_\beta+\left(d_\beta\right)\mathcal{F}^\circ_{g_\beta}, \operatorname{max}\{r_\beta,r_\gamma\} \right]\right).
      \end{split}
    \end{equation*}

    Since $\gamma\mathfrak{C}_{\mathbf{mg}}\alpha$, by Theorem~\ref{theorem-1.11} we have that $g_\gamma=g_\alpha$ and $r_\gamma - \left(d_\gamma\right)\mathcal{F}^\circ_{g_\gamma} = r_\alpha - \left(d_\alpha\right)\mathcal{F}^\circ_{g_\alpha}$, then for any $x\in \kappa$
    \begin{equation*}
      \begin{split}
        \left(x\right)&\left(\operatorname{max}\{r_\beta,r_\gamma\}\right) =
        \left(x\right)\left(\operatorname{max}\{r_\alpha - \operatorname{min}\{\left(d_\alpha\right)\mathcal{F}^\circ_{g_\alpha}, r_\alpha \} + \mathbf{1},r_\gamma\}\right) =\\
        =&
        \begin{cases}
          \operatorname{max}\{\left(x\right)r_\alpha - \left(x\right)r_\alpha + \left(x\right)\mathbf{1},\left(x\right)r_\gamma\}, &\text{if } \left(x\right)\left(d_\alpha\right)\mathcal{F}^\circ_{g_\alpha} > \left(x\right)r_\alpha\\
          \operatorname{max}\{\left(x\right)r_\alpha - \left(x\right)\left(d_\alpha\right)\mathcal{F}^\circ_{g_\alpha} + \left(x\right)\mathbf{1},\left(x\right)r_\gamma\}, &\text{otherwise} 
        \end{cases}
        =\\
        =&
        \begin{cases}
          \operatorname{max}\{1,\left(x\right)r_\gamma\}, &\text{if } \left(x\right)\left(d_\alpha\right)\mathcal{F}^\circ_{g_\alpha} > \left(x\right)r_\alpha\\
          \operatorname{max}\{\left(x\right)r_\gamma - \left(x\right)\left(d_\gamma\right)\mathcal{F}^\circ_{g_\gamma} + 1,\left(x\right)r_\gamma\}, &\text{otherwise} 
        \end{cases}=\\
        =&\left(x\right)r_\gamma,
      \end{split}
    \end{equation*}
    so $\operatorname{max}\{r_\beta,r_\gamma\}=r_\gamma$. Also
    \begin{equation*}
      \begin{split}
        \operatorname{max}\{r_\beta,r_\gamma\}-r_\beta+\left(d_\beta\right)\mathcal{F}^\circ_{g_\beta}&=
        r_\gamma-r_\beta+\left(d_\beta\right)\mathcal{F}^\circ_{g_\beta}=\\
        &=r_\gamma - r_\alpha + \left(d_\alpha\right)\mathcal{F}^\circ_{g_\alpha}=\\
        &=\left(d_\gamma\right)\mathcal{F}^\circ_{g_\gamma},
      \end{split}
    \end{equation*}
    so
    \begin{equation*}
    \left(g_\beta,\left[\operatorname{max}\{r_\beta,r_\gamma\}-r_\beta+\left(d_\beta\right)\mathcal{F}^\circ_{g_\beta}, \operatorname{max}\{r_\beta,r_\gamma\} \right]\right)=\left(g_\gamma,\left[\left(d_\gamma\right)\mathcal{F}^\circ_{g_\gamma}, r_\gamma \right]\right)=\left(\gamma\right)\Psi.
    \end{equation*}

    The equality $\left(\beta\right)\Psi\left(\varepsilon\right)\Psi=\left(\gamma\right)\Psi$ implies that $\gamma=\beta\varepsilon$, so $\gamma\preccurlyeq\beta$. This means that the element $\beta$ is the biggest element in the $\mathfrak{C}_{\mathbf{mg}}$-class of the element $\alpha$ in $\mathcal{IP\!F}\left(\sigma{\mathbb{N}^\kappa}\right)$.
    \end{proof}

    \begin{lemma}\label{lemma-2.20}
      Let $\kappa$ be any infinite cardinal and let $\mathfrak{C}$ be a congruence on the semigroup $\mathcal{IP\!F}\left(\sigma{\mathbb{N}^\kappa}\right)$ such that $\varepsilon\mathfrak{C}\iota$ for some two distinct idempotents $\varepsilon,\iota\in\mathcal{IP\!F}\left(\sigma{\mathbb{N}^\kappa}\right)$. Then $\varsigma\mathfrak{C}\upsilon$ for all idempotents $\varsigma,\upsilon$ of $\mathcal{IP\!F}\left(\sigma{\mathbb{N}^\kappa}\right)$.
      \end{lemma}

      \begin{proof}[\textsl{Proof}]
        We observe that without loss of generality we may assume that $\varepsilon\preccurlyeq\iota$ where $\preccurlyeq$ is the natural partial order on the semilattice $E(\mathcal{IP\!F}\left(\sigma{\mathbb{N}^\kappa}\right))$. Indeed, if $\varepsilon,\iota\in E(\mathcal{IP\!F}\left(\sigma{\mathbb{N}^\kappa}\right))$ then $\varepsilon\mathfrak{C}\iota$ implies that $\varepsilon=\varepsilon\varepsilon\mathfrak{C}\iota\varepsilon$, and since the idempotents $\varepsilon$ and $\iota$ are distinct in $\mathcal{IP\!F}\left(\sigma{\mathbb{N}^\kappa}\right)$ we have that $\iota\varepsilon\preccurlyeq\varepsilon$.

        Now, the inequality $\varepsilon\preccurlyeq\iota$ implies that $\operatorname{dom}\varepsilon\subseteq \operatorname{dom}\iota$. 
        Next, we define partial map $\alpha\colon\sigma{\mathbb{N}^\kappa}\rightharpoonup\sigma{\mathbb{N}^\kappa}$ in the following way:
        \begin{equation*}
          \operatorname{dom}\alpha=\sigma{\mathbb{N}^\kappa}, \qquad \operatorname{ran}\alpha=\operatorname{dom}\iota \qquad \hbox{and}\qquad \left(z\right)\alpha=z+d_\iota-\mathbf{1}, \quad \hbox{for any} \quad z\in\operatorname{dom}\alpha.
        \end{equation*}
        The definition of $\alpha$ implies that $\alpha\iota\alpha^{-1}=\alpha\alpha^{-1}=\mathbb{I}$ and $\alpha^{-1}\alpha=\iota$, and moreover, we have that
        \begin{align*}
          \left(\alpha\varepsilon\alpha^{-1}\right)\left(\alpha\varepsilon\alpha^{-1}\right)&  =\alpha\varepsilon\left(\alpha^{-1}\alpha\right)\varepsilon\alpha^{-1}= \\
           &=\alpha\varepsilon\iota\varepsilon\alpha^{-1}=\\
           &=\alpha\varepsilon\varepsilon\alpha^{-1}=\\
           &=\alpha\varepsilon\alpha^{-1},
        \end{align*}
        which implies that $\alpha\varepsilon\alpha^{-1}$ is an idempotent of $\mathcal{IP\!F}\left(\sigma{\mathbb{N}^\kappa}\right)$ such that $\alpha\varepsilon\alpha^{-1}\neq \mathbb{I}$.

        Thus, it was shown that there exists a non-unit idempotent $\varepsilon^*$ in $\mathcal{IP\!F}\left(\sigma{\mathbb{N}^\kappa}\right)$ such that $\varepsilon^*\mathfrak{C}\mathbb{I}$. This implies that $\varepsilon_0\mathfrak{C}\mathbb{I}$ for any idempotent $\varepsilon_0$ of $\mathcal{IP\!F}\left(\sigma{\mathbb{N}^\kappa}\right)$ such that $\varepsilon^*\preccurlyeq\varepsilon_0\preccurlyeq\mathbb{I}$. Since $\varepsilon^* \neq \mathbb{I}$ we have that $d_{\varepsilon^*} \neq \mathbf{1}$, so there exists $x\in\kappa$ such that $\left(x\right)d_{\varepsilon^*} \neq 1$, thus $2_x\leqslant d_{\varepsilon^*}$. Consider an idempotent $\varepsilon_x$ in $\mathcal{IP\!F}\left(\sigma{\mathbb{N}^\kappa}\right)$ such that $d_{\varepsilon_x} = 2_x$. Then $d_{\varepsilon_x} = 2_x\leqslant d_{\varepsilon^*}$ implies that $\varepsilon^*\preccurlyeq\varepsilon_x$, so $\varepsilon_x\mathfrak{C}\mathbb{I}$.

        Fix an arbitrary $y\in\kappa\setminus\{x\}$. Define a bijection on the set $\kappa$ in the following way:
        \begin{equation*}
          \left(x\right)g=y, \qquad \left(y\right)g=x \qquad \hbox{and}\qquad \left(t\right)g=t, \quad \hbox{for } \quad t\in\kappa\setminus\{x,y\}.
        \end{equation*}
        Next, consider the map $\mathcal{F}^\circ_g$ as an element of $\mathcal{IP\!F}\left(\sigma{\mathbb{N}^\kappa}\right)$. The definition of $g$ implies that $g^{-1}=g$, then, by Lemma~\ref{lemma-1.2-3}$\left(i\right)$ we have that $\left(\mathcal{F}^\circ_g\right)^{-1}=\mathcal{F}^\circ_{g^{-1}}=\mathcal{F}^\circ_g$ and then
        \begin{equation*}
        \mathcal{F}^\circ_g\mathbb{I}\mathcal{F}^\circ_g=\mathcal{F}^\circ_g\mathcal{F}^\circ_g=\mathcal{F}^\circ_g\left(\mathcal{F}^\circ_g\right)^{-1} = \mathbb{I}.
        \end{equation*}
        The calculations
        \begin{equation*}
          \begin{split}
            \left(\mathcal{F}^\circ_g \varepsilon_x \mathcal{F}^\circ_g\right)\Psi=&
            \left(\mathcal{F}^\circ_g\right)\Psi\left(\varepsilon_{d_x}\right)\Psi\left(\mathcal{F}^\circ_g\right)\Psi=\\
            &=\left(g,\left[\mathbf{1}, \mathbf{1}\right]\right) \left(id_\kappa,\left[2_x, 2_x\right]\right) \left(g,\left[\mathbf{1}, \mathbf{1}\right]\right)=\\
            &=\left(g,\left[2_x, 2_x\right]\right) \left(g,\left[\mathbf{1}, \mathbf{1}\right]\right) =\\
            &=\left(gg,\left[\left(2_x\right)\mathcal{F}^\circ_g, \left(2_x\right)\mathcal{F}^\circ_g\right]\right)=\\
            &=\left(id_\kappa,\left[2_{\left(x\right)g}, 2_{\left(x\right)g}\right]\right)=\\
            &=\left(id_\kappa,\left[2_y, 2_y\right]\right)=\\
            &=\left(\varepsilon_y\right)\Psi
          \end{split}
        \end{equation*}
        shows that $\mathcal{F}^\circ_g \varepsilon_x \mathcal{F}^\circ_g = \varepsilon_y$, where $\varepsilon_y$ is an idempotent in $\mathcal{IP\!F}\left(\sigma{\mathbb{N}^\kappa}\right)$ such that $d_{\varepsilon_y} = 2_y$. Then
        \begin{equation*}
          \varepsilon_y=\left(\mathcal{F}^\circ_g \varepsilon_x \mathcal{F}^\circ_g\right) \mathfrak{C} \left(\mathcal{F}^\circ_g\mathbb{I}\mathcal{F}^\circ_g\right)=\mathbb{I}
        \end{equation*}
        implies that $\varepsilon_y \mathfrak{C} \mathbb{I}$.

      The above arguments imply that $\varepsilon_x\mathfrak{C}\mathbb{I}$ for every idempotent $\varepsilon_x\in\mathcal{IP\!F}\left(\sigma{\mathbb{N}^\kappa}\right)$ such that $\varepsilon_x$ is the identity map of the principal filter ${\uparrow}2_x$ of the poset $\left(\sigma{\mathbb{N}^\kappa}, \leqslant\right)$, $x\in\kappa$. Now, fix an idempotent $\zeta$ in $\mathcal{IP\!F}\left(\sigma{\mathbb{N}^\kappa}\right)$ and consider the set $A=\{ x\in\kappa \mid \left(x\right)d_\zeta \neq 1 \}$. Since $d_\zeta\in\sigma{\mathbb{N}^\kappa}$ the set $A$ is finite, so there exists $k\in\mathbb{N}$ such that $A=\{x_1, x_2,\dots,x_k\}$ for some $x_1, x_2,\dots,x_k\in\kappa$. Consider the idempotent $\varepsilon_A=\varepsilon_{x_1}\dots\varepsilon_{x_k}$. Since $\mathfrak{C}$ is congruence, $\varepsilon_{x_i}\mathfrak{C}\mathbb{I}$ for any $x_i\in A$ and $A$ is finite we have that $\left(\varepsilon_{x_1}\dots\varepsilon_{x_k}\right)\mathfrak{C}\mathbb{I}$. The definition of $\varepsilon_A$ and the semigroup operation of $\mathcal{IP\!F}\left(\sigma{\mathbb{N}^\kappa}\right)$ imply that $d_{\varepsilon_A}=2_A$, where
      \begin{equation*}
        \left(t\right)2_A =
         \begin{cases}
           2 &\text{if $t\in A$}\\
           1 &\text{otherwise.}
         \end{cases}
        \end{equation*}
        We define the partial map $\gamma\colon \sigma{\mathbb{N}^\kappa}\rightharpoonup \sigma{\mathbb{N}^\kappa}$ in the following way:
        \begin{equation*}
          \operatorname{dom}\gamma=\sigma{\mathbb{N}^\kappa}, \quad \operatorname{ran}\gamma={\uparrow}2_A \quad  \hbox{~and~} \quad \left(z\right)\gamma=z+2_A-\mathbf{1}, \quad  \hbox{~for any~} \quad z\in\operatorname{dom}\gamma.
        \end{equation*}
        The definition of $\gamma$ implies that that $\gamma\gamma^{-1}=\mathbb{I}$ and $\gamma^{-1}\gamma=\varepsilon_A$. For any positive integer $n\in\mathbb{N}$ consider the idempotent
        \begin{equation*}
        \left(\gamma^{-1}\right)^n\gamma^n = \underbrace{\gamma^{-1}\dots\gamma^{-1}}_{n\hbox{\footnotesize{-times}}} \underbrace{\gamma\dots\gamma}_{n\hbox{\footnotesize{-times}}}.
        \end{equation*}
        Since $\varepsilon_A = \gamma^{-1}\gamma\mathfrak{C}\mathbb{I}$ we have that $\gamma^{-1}\gamma^{-1}\gamma\gamma\mathfrak{C}\gamma^{-1}\gamma=\varepsilon_A$ and $\gamma^{-1}\gamma^{-1}\gamma\gamma\mathfrak{C}\mathbb{I}$, so by induction $\left(\gamma^{-1}\right)^n\gamma^n\mathfrak{C}\mathbb{I}$, for any $n\in\mathbb{N}$. Also, by induction, we have that $d_{\left(\gamma^{-1}\right)^n\gamma^n}=(n+1)_A$, where
        \begin{equation*}
          \left(t\right)(n+1)_A =
           \begin{cases}
             n+1 &\text{if $t\in A$}\\
             1 &\text{otherwise,}
           \end{cases}
          \end{equation*}
        for any $n\in\mathbb{N}$. Thus, we have that
        \begin{equation*}
        d_\zeta \leqslant d_{\left(\gamma^{-1}\right)^m\gamma^m} = (m+1)_A,
        \end{equation*}
        where $m~=~\operatorname{max}\{\left(x\right)d_\zeta \mid x\in\kappa\}$, implies that $\left(\gamma^{-1}\right)^m\gamma^m\preccurlyeq \zeta$, so $\zeta \mathfrak{C}\mathbb{I}$.
      \end{proof}

    \begin{lemma}\label{lemma-2.21}
      Let $\kappa$ be any infinite cardinal and let $\mathfrak{C}$ be a congruence on the semigroup $\mathcal{IP\!F}\left(\sigma{\mathbb{N}^\kappa}\right)$ such that $\alpha\mathfrak{C}\beta$ for some non-$\mathscr{H}$-equivalent elements $\alpha,\beta\in\mathcal{IP\!F}\left(\sigma{\mathbb{N}^\kappa}\right)$. Then $\varepsilon\mathfrak{C}\iota$ for all idempotents $\varepsilon,\iota$ of $\mathcal{IP\!F}\left(\sigma{\mathbb{N}^\kappa}\right)$.
      \end{lemma}

      \begin{proof}[\textsl{Proof}]
      Since $\alpha$ and $\beta$ are not-$\mathscr{H}$-equivalent in $\mathcal{IP\!F}\left(\sigma{\mathbb{N}^\kappa}\right)$ we have that either $\alpha\alpha^{-1}\neq\beta\beta^{-1}$ or $\alpha^{-1}\alpha\neq\beta^{-1}\beta$ (see \cite[p.~82]{Lawson-1998}). Then Proposition~4 from \cite[Section~2.3]{Lawson-1998} implies that $\alpha\alpha^{-1}\mathfrak{C}\beta\beta^{-1}$ and $\alpha^{-1}\alpha\mathfrak{C}\beta^{-1}\beta$ and hence the assumption of Lemma~\ref{lemma-2.20} holds.
      \end{proof}

    \begin{lemma}\label{lemma-2.22}
      Let $\kappa$ be any infinite cardinal and let $\mathfrak{C}$ be a congruence on the semigroup $\mathcal{IP\!F}\left(\sigma{\mathbb{N}^\kappa}\right)$ such that $\alpha\mathfrak{C}\beta$ for some two distinct $\mathscr{H}$-equivalent elements $\alpha,\beta\in\mathcal{IP\!F}\left(\sigma{\mathbb{N}^\kappa}\right)$. Then $\varepsilon\mathfrak{C}\iota$ for all idempotents $\varepsilon,\iota$ of $\mathcal{IP\!F}\left(\sigma{\mathbb{N}^\kappa}\right)$.
    \end{lemma}

        \begin{proof}[\textsl{Proof}]
        By Proposition~\ref{proposition-2.1}$\left(vi\right)$ the semigroup $\mathcal{IP\!F}\left(\sigma{\mathbb{N}^\kappa}\right)$ is simple and then Theorem~2.3 from \cite{Clifford-Preston-1961-1967} implies that there exist $\mu,\xi\in\mathcal{IP\!F}\left(\sigma{\mathbb{N}^\kappa}\right)$ such that $f\colon H_{\alpha}\to H_{\mathbb{I}}\colon \chi\mapsto\mu\chi\xi$ maps $\alpha$ to $\mathbb{I}$ and $\beta$ to $\gamma\neq\mathbb{I}$, respectively, which implies that $\mathbb{I}\mathfrak{C}\gamma$. Since $\gamma$ is an element of the group of units of the semigroup $\mathcal{IP\!F}\left(\sigma{\mathbb{N}^\kappa}\right)$, by Theorem~\ref{theorem-1.3}, $\gamma=\mathcal{F}^\circ_{g_\gamma}$ and since $\gamma\neq\mathbb{I}$ we have that $g_\gamma\neq id_\kappa$, so there exists $x\in\kappa$ such that $\left(x\right)g_\gamma\neq x$. Put $\varepsilon$ as the identity map with $d_\varepsilon=2_x$. Since $\mathfrak{C}$ is a congruence on the semigroup $\mathcal{IP\!F}\left(\sigma{\mathbb{N}^\kappa}\right)$ and $\gamma\in H_{\mathbb{I}}$ we have that
        \begin{equation*}
          \varepsilon=\varepsilon\varepsilon=\varepsilon\mathbb{I}\varepsilon\mathfrak{C}\varepsilon\gamma\varepsilon.
        \end{equation*}
        Proposition~\ref{proposition-1.11} implies that
        \begin{equation*}
        \left(\varepsilon\gamma\varepsilon\right)\Psi=\left(g_\gamma,\left[\operatorname{max}\{\left(2_x\right)\mathcal{F}^\circ_{g_\gamma}, 2_x\}, \operatorname{max}\{\left(2_x\right)\mathcal{F}^\circ_{g_\gamma}, 2_x\}\right]\right).
        \end{equation*}
        By Lemma~\ref{lemma-1.2-3}$\left(v\right)$ we have that $\left(2_x\right)\mathcal{F}^\circ_{g_\gamma}=2_{\left(x\right)g_\gamma}\neq 2_x$, this and the definition of elements $2_x$ and $2_{\left(x\right)g_\gamma}$ imply that $\operatorname{max}\{\left(2_x\right)\mathcal{F}^\circ_{g_\gamma}, 2_x\}\neq 2_x$, so
        \begin{equation*}
        r_{\varepsilon\gamma\varepsilon}=\operatorname{max}\{\left(2_x\right)\mathcal{F}^\circ_{g_\gamma}, 2_x\}\neq 2_x = r_{\varepsilon},
        \end{equation*}
        then by Proposition~\ref{proposition-2.1}$\left(v\right)$, $\varepsilon\gamma\varepsilon$ and $\varepsilon$ are non-$\mathscr{H}$-equivalent elements in $\mathcal{IP\!F}\left(\sigma{\mathbb{N}^\kappa}\right)$. Next, we apply Lemma~\ref{lemma-2.21}.
        \end{proof}

    \begin{theorem}\label{theorem-2.23}
      For any infinite cardinal $\kappa$ every non-identity congruence $\mathfrak{C}$ on the semigroup $\mathcal{IP\!F}\left(\sigma{\mathbb{N}^\kappa}\right)$ is group.
      \end{theorem}

      \begin{proof}[\textsl{Proof}]
      For every non-identity congruence $\mathfrak{C}$ on $\mathcal{IP\!F}\left(\sigma{\mathbb{N}^\kappa}\right)$ there exist two distinct elements $\alpha,\beta\in\mathcal{IP\!F}\left(\sigma{\mathbb{N}^\kappa}\right)$ such that $\alpha\mathfrak{C}\beta$. If $\alpha\mathscr{H}\beta$ in $\mathcal{IP\!F}\left(\sigma{\mathbb{N}^\kappa}\right)$ then by Lemma~\ref{lemma-2.21} all idempotents of the semigroup $\mathcal{IP\!F}\left(\sigma{\mathbb{N}^\kappa}\right)$ are $\mathfrak{C}$-equivalent, otherwise by Lemma~\ref{lemma-2.22} we get the same. Thus, by Lemma~II.1.10 of \cite{Petrich-1984} the quotient semigroup $\mathcal{IP\!F}\left(\sigma{\mathbb{N}^\kappa}\right)/\mathfrak{C}$ has a unique idempotent and hence it is a group.
      \end{proof}

  \section*{\textbf{Acknowledgements}}

  The author acknowledges Oleg Gutik and Alex Ravsky for their useful comments and suggestions.


\begin{thebibliography}{W}

    \bibitem{Andersen-1952}
    O.~Andersen,
    \textit{Ein Bericht \"{u}ber die Struktur abstrakter Halbgruppen},
    PhD Thesis, Hamburg, 1952.
  
    \bibitem{Anderson-Hunter-Koch-1965}
    L.~W.~Anderson, R.~P.~Hunter, and R.~J.~Koch,
    \textit{Some results on stability in semigroups}.
    Trans. Amer. Math. Soc. {\bf 117} (1965), 521--529.
    DOI: 10.2307/1994222
  
    \bibitem{Banakh-Dimitrova-Gutik-2009}
    T.~Banakh, S.~Dimitrova, and O.~Gutik,
    \textit{The Rees-Suschkewitsch Theorem for simple to\-po\-lo\-gical semigroups},
    Mat. Stud. \textbf{31} (2009), no.~2, 211--218.
  
    \bibitem{Banakh-Dimitrova-Gutik-2010}
    T.~Banakh, S.~Dimitrova, and O.~Gutik,
    \textit{Embedding the bicyclic semigroup into countably compact topological semigroups},
    Topology Appl. \textbf{157} (2010), no.~18, 2803--2814. \\
    DOI: 10.1016/j.topol.2010.08.020
  
    \bibitem{Bardyla-Gutik-2016}
    S.~Bardyla and O.~Gutik,
    \textit{On a semitopological polycyclic monoid},
    Algebra Discr. Math. \textbf{21} (2016), no. 2, 163--183.
  
    \bibitem{Bardyla-Gutik-2020}
    S.~Bardyla and O.~Gutik,
    \textit{On the lattice of weak topologies on the bicyclic monoid with adjoined zero},
    Algebra Discr. Math. \textbf{30} (2020), no. 1, 26--43.
    DOI: 10.12958/adm1459
  
    \bibitem{Bertman-West-1976}
    M.~O.~Bertman and T.~T.~West,
    \textit{\it Conditionally compact bicyclic semitopological semigroups},
    Proc. Roy. Irish Acad. Sec. A {\bf76} (1976), no.~21--23, 219--226.
  
    \bibitem{Chuchman-Gutik-2010}
    I.~Chuchman and O.~Gutik,
    \textit{Topological monoids of almost monotone injective co-finite partial selfmaps of the set of positive integers},
    Carpathian Math. Publ. 2 (2010), no. 1, 119--132.
  
    \bibitem{Clifford-Preston-1961-1967}
    A.~H.~Clifford and G.~B.~Preston,
    \textit{The Algebraic Theory of Semigroups}, Vols. I and II,
    Amer. Math. Soc. Surveys {\bf 7}, Pro\-vidence, R.I., 1961 and 1967.
  
    \bibitem{Gutik-Lysetska-2023}
    O.~Gutik and O.~Lysetska,
    \textit{On the semigroup $\boldsymbol{B}_\omega^{\mathscr{F}}$ which is generated by the family $\mathscr{F}$ of atomic subsets of $\omega$},
    Visn. L'viv. Univ., Ser. Mekh.-Mat. 92 (2021), 34--50. \\
    DOI: 10.30970/vmm.2021.92.034-050
  
    \bibitem{Gutik-Khylynskyi-2022}
    O.~Gutik and P.~Khylynskyi,
    \textit{On a locally compact monoid of cofinite partial isometries of $\mathbb{N}$ with adjoined zero},
    Topol. Algebra Appl. \textbf{10} (2022), no. 1, 233--245. \\
    DOI: 10.1515/taa-2022-0130
  
    \bibitem{Gutik-Krokhmalna-2019}
    O.~Gutik and O.~Krokhmalna,
    \textit{The monoid of monotone injective partial selfmaps of the poset $(\mathbb{N}^{3},\leqslant)$ with cofinite domains and images},
    Visn. L'viv. Univ., Ser. Mekh.-Mat. \textbf{88} (2019), 32--50.
  
    \bibitem{Gutik-Maksymyk-2016}
    O.~Gutik and K.~Maksymyk,
    \textit{On semitopological bicyclic extensions of linearly ordered groups},
    Mat. Metody Fiz.-Mekh. Polya \textbf{59} (2016), no. 4, 31--43; \textbf{reprinted version}: J. Math. Sci. \textbf{238} (2019), no. 1, 32--45.
    DOI: 10.1007/s10958-019-04216-x
  
    \bibitem{Gutik-Maksymyk-2016-2}
    O.~Gutik and K.~Maksymyk,
    \textit{On semitopological interassociates of the bicyclic monoid},
    Visn. L'viv. Univ., Ser. Mekh.-Mat. \textbf{82} (2016), 98--108.
  
    \bibitem{Gutik-Mokrytskyi-2020}
    O.~Gutik and T.~Mokrytskyi,
    \textit{The monoid of order isomorphisms between principal filters of ${\mathbb{N}}^n$},
    Eur. J. Math. \textbf{6} (2020), no. 1, 14--36.
    DOI: 10.1007/s10958-019-04216-x
  
    \bibitem{Gutik-Mykhalenych-2020}
    O.~Gutik and M.~Mykhalenych,
    \textit{On some generalization of the bicyclic monoid},
    Visn. L'viv. Univ., Ser. Mekh.-Mat. \textbf{90} (2020), 5--19 (in Ukrainian). \\
    DOI: 10.30970/vmm.2020.90.005-019
  
    \bibitem{Gutik-Popadiuk-2022}
    O.~V. Gutik and O.~B. Popadiuk,
    \textit{On the semigroup of injective endomorphisms of the semigroup $B_\omega^{\mathscr{F}_n}$ which is generated by the family $\mathscr{F}_n$ of finite bounded intervals of $\omega$},
    Mat. Metody Fiz.-Mekh. Polya \textbf{65} (2022), no. 1--2, 42--57.
  
    \bibitem{Gutik-Pozdniakova-2022}
    O.~Gutik and I.~Pozdniakova,
    \textit{On the group of automorphisms of the semigroup $B_{\mathbb{Z}}^{\mathscr{F}}$ with the family $\mathscr{F}$ of inductive nonempty subsets of $\omega$}, Algebra Discr. Math. (to appear) (arXiv:2206.12819).
  
    \bibitem{Gutik-Repovs-2007}
    O.~Gutik and D.~Repov\v{s},
    \textit{On countably compact $0$-simple topological inverse semigroups},
    Semigroup Forum \textbf{75} (2007), no.~2, 464--469.
    DOI: 10.1007/s00233-007-0706-x
  
    \bibitem{Gutik-Savchuk-2018}
    O.~Gutik and A.~Savchuk,
    \textit{The semigroup of partial co-finite isometries of positive integers},
    Bukovyn. Mat. Zh. \textbf{6} (2018), no.1--2, 42--51 (in Ukrainian).
    DOI: 10.31861/bmj2018.01.042
  
    \bibitem{Gutik-Savchuk-2019}
    O.~Gutik and A.~Savchuk,
    \textit{On the monoid of cofinite partial isometries of $\mathbb{N}$ with the usual metric},
    Visn. L'viv. Univ., Ser. Mekh.-Mat. \textbf{89} (2020) 17--30. \\
    DOI: 10.30970/vmm.2020.89.017-030
  
    \bibitem{Hildebrant-Koch-1986}
    J.~A.~Hildebrant and R.~J.~Koch,
    \textit{\it Swelling actions of $\Gamma$-compact semigroups},
    Se\-mi\-group Fo\-rum {\bf 33} (1986), no.~1,  65--85.
    DOI: 10.1007/BF02573183
  
    \bibitem{Lawson-1998}
    M.~Lawson,
    \textit{Inverse semigroups. The theory of partial symmetries},
    World Scientific, Sin\-gapore, 1998.
  
    \bibitem{McFadden-Carroll-1971}
    R. McFadden and L. O'Carroll,
    \textit{$F$-inverse semigroups},
    Proc. Lond. Math. Soc., III Ser. \textbf{22} (1971), no.~4, 652--666.
    DOI:  10.1112/plms/s3-22.4.652
  
    \bibitem{Mokrytskyi-2019}
    T.~Mokrytskyi,
    \textit{On the dichotomy of a locally compact semitopological monoid of order isomorphisms between principal filters of ${\mathbb{N}}^n$ with adjoined zero},
    Visn. Lviv Univ., Ser. Mekh.-Mat. \textbf{87} (2019), 37--45.
    DOI:  10.30970/vmm.2019.87.037-045
  
    \bibitem{Munn-1966}
    W.~D.~Munn,
    \textit{Uniform semilattices and bisimple inverse semigroups},
    Q. J. Math., Oxf. II. Ser. \textbf{17} (1966), no.~1, 151--159.
    DOI: 10.1093/qmath/17.1.151
  
    \bibitem{Petrich-1984}
    M.~Petrich,
    \textit{Inverse semigroups},
    John Wiley $\&$ Sons, New York, 1984.
  
    \bibitem{Saito-1965}
    T. Sait\^{o},
    \textit{Proper ordered inverse semigroups},
    Pacif. J. Math. \textbf{15}  (1965), no.~2, 649--666. \\
    DOI: 10.2140/pjm.1965.15.649
  
    \bibitem{Vagner-1952}
    V.~V.~Wagner,
    \textit{Generalized groups},
    Dokl. Akad. Nauk SSSR \textbf{84} (1952), 1119--1122 (in Russian).
  
  \end{thebibliography}
  \end{document}